\documentclass[12pt]{article}

\usepackage[english]{babel}
\usepackage{upref,amsfonts,amsxtra,latexsym,color}
\usepackage{amssymb,amsmath,amsthm,a4wide,epsf,mathrsfs,verbatim,hyperref}
\usepackage[all]{xy}

\newtheorem{theorem}{Theorem}[subsection]
\newtheorem{lemma}[theorem]{Lemma}
\newtheorem{corollary}[theorem]{Corollary}
\newtheorem{proposition}[theorem]{Proposition}

\newtheorem{conjecture}[theorem]{Conjecture}

\theoremstyle{definition}
\newtheorem{definition}[theorem]{Definition}
\newtheorem{remark}[theorem]{Remark}
\newtheorem{example}[theorem]{Example}

\numberwithin{equation}{section}
\numberwithin{theorem}{section}

\newcommand{\tensorhat}{\, \widehat{\otimes} \,}
\newcommand{\otensorhat}{\, \widehat{\odot} \,}

\newcommand{\rank}{\mathrm{rank}}

\newcommand{\Tr}{\mathrm{Tr}}
\newcommand{\tr}{\mathrm{tr}}

\newcommand{\sa}{\mathrm{sa}}
\newcommand{\mx}{\mathrm{max}}
\newcommand{\aff}{\mathrm{aff}}

\newcommand{\cA}{{\cal A}}

\newcommand{\cB}{{\cal B}}
\newcommand{\cE}{{\cal E}}

\newcommand{\cC}{{\cal C}}

\newcommand{\C}{{\mathbb C}}

\newcommand{\N}{{\mathbb N}}

\newcommand{\R}{{\mathbb R}}

\newcommand{\Cs}{{$C^*$-al\-ge\-bra}}

\newcommand{\id}{\mathrm{id}}

\newcommand{\conv}{{\mathrm{conv}}}
\newcommand{\sh}{{$^*$-ho\-mo\-mor\-phism}}

\date{}
\title{Entanglement in $C^*$-algebras: tensor products of state spaces}

\author{Magdalena Musat and Mikael R\o rdam}

\begin{document}

\maketitle

\begin{abstract} 
We analyze the Namioka--Phelps minimal and maximal tensor products of compact convex sets which arise as the state spaces of unital \Cs s. Relatedly,  we study entanglement in (infinite dimensional) \Cs s. 
While the \emph{minimal} Namioka--Phelps tensor product of the state spaces of two \Cs s is  the well-known set of separable (= un-entangled) states on the (minimal) tensor product of the \Cs s, we also describe  the more elusive \emph{maximal}  Namioka--Phelps tensor product of state spaces of \Cs s. We show that the minimal and maximal tensor products of state spaces of \Cs s agree precisely when one of the two \Cs s is commutative,  which confirms Barker's conjecture in the case where the compact convex sets are state
 spaces of \Cs s. 

Further, the Namioka--Phelps tensor product of the trace simplexes of two or more unital \Cs s  is shown to be the trace simplex of the (minimal or maximal) tensor product of the \Cs s. This enables a systematic way of determining the trace simplex of a tensor product of \Cs s.
\end{abstract}

\section{Introduction}

\noindent Namioka and Phelps introduced in their paper, \cite{NP:PJM1969}, two different ways in which one can associate a tensor product of two convex compact sets $K_1$ and $K_2$: a minimal and a maximal one, here denoted by $K_1 \otimes_* K_2$ and $K_1 \otimes^* K_2$, respectively. The two notions depend on two different ways in which one can define a positive cone in the (algebraic) tensor product of the ordered vector spaces of affine functions on $K_1$ and $K_2$: an ``inner'' positive cone generated by positive elementary tensors, and an ``outer'' positive cone consisting of elements  that are positive when paired with a tensor product of positive functionals. Namioka and Phelps proved that the two tensor products agree when one of the two convex compact sets is a Choquet simplex. The converse, if $K_1 \otimes_* K_2 \ne K_1 \otimes^* K_2$, whenever neither $K_1$ nor $K_2$ are Choquet simplexes, is known as Barker's conjecture. (Barker conjectured in  \cite{Barker:tensor}, see also his  survery paper \cite{Barker:cones}, that the analogous statement should hold for finite dimensional cones.) Namioka and Phelps showed that Barker's conjecture holds when one of the two convex compact sets is a square, and Aubrun--Lami--Palazuelos--Plavala, \cite{ALPP:GAFA2021}, recently settled Barker's conjecture in its original formulation, that is, for finite dimensional cones, and hence also for finite dimensional compact convex sets.

The overall purpose of this paper is to analyze the Namioka--Phelps tensor products of the state spaces  of two \Cs s, and to relate these to entanglement phenomena in (infinite dimensional) \Cs s.  
More specifically, given two unital \Cs s $\cA$ and $\cB$, 
we describe the Namioka--Phelps tensor products of the state spaces $S(\cA)$ and $S(\cB)$ as  sets of functionals on the tensor products  of $\cA$ and $\cB$ (algebraic, spatial and the maximal). 
The minimal tensor product $S(\cA) \otimes_* S(\cB)$ is precisely the set of separable (= un-entangled) states on the spatial (minimal) tensor product $\cA \otimes \cB$, thus connecting the Namioka--Phelps tensor products to entanglement in quantum information theory. The maximal tensor product, $S(\cA) \otimes^* S(\cB)$, is more elusive, but can be understood in terms of unital positive maps, cf.\ Proposition~\ref{prop:max-tensor-matrix} and Remark~\ref{rem:max-positive}.  
Both tensor products of state spaces are well understood in the finite dimensional case, cf.\ Remark~\ref{rem:entangled-matrices} and Theorem~\ref{prop:kappa(n,m)}. 

A positive element  in the (spatial or maximal) tensor product $\cA \otimes \cB$ is entangled if it is not separable, i.e., belongs to the closed cone $\cA^+ \otimes \cB^+$ generated by elementary positive elements. In the finite dimensional case, entanglement of positive elements (of trace 1) and entanglement of states is the same (via density matrices), but this is no longer the case for (infinite dimensional) \Cs s. We consider both kinds of entanglement, and establish permanence properties of entanglement with respect to subalgebras and quotients. 

We remark that the maximal tensor product $K_1 \otimes^* K_2$  of two \emph{finite dimensional} convex compact sets $K_1$ and $K_2$ is bounded relative to the minimal tensor product $K_1 \otimes_* K_2$, and that this no longer holds when $K_1$ and $K_2$ are the state spaces of sufficiently non-commutatative (more precisely, non sub-homogeneous) \Cs s. We introduce a numerical constant $\kappa(\cA,\cB)$, for each pair of \Cs s $\cA$ and $\cB$, that measures the relative sizes of $S(\cA \otimes_\mx \cB) \subseteq S(\cA) \otimes^* S(\cB)$. We determine  this constant when $\cA$ and $\cB$ are matrix algebras and show in general that it is bounded from below by the minimum rank of $\cA$ and $\cB$.

As a summary of our investigations, we show in Theorem~\ref{thm:twoabelian} that the three sets $S(\cA) \otimes_* S(\cB)$, $S(\cA \otimes \cB)$ and $S(\cA) \otimes^* S(\cB)$ are equal if and only if one of $\cA$ and $\cB$ is commutative, and that none of these sets are pairwise equal otherwise. Also, $\cA^+ \otimes \cB^+$ and $(\cA \otimes \cB)^+$ are equal if and only if one of $\cA$ and $\cB$ is commutative. As a byproduct, we reprove  and put into the context of tensor products of convex compact sets and entanglement, the well-known fact that the state space of a \Cs{} is a Choquet simplex if and only if the \Cs{} is commutative.
As a consequence, this confirms Barker's conjecture for convex compact sets arising as the state space of \Cs s, and leaves open if this remains true for all (infinite dimensional) compact convex sets; this question arose in a BIRS (Granada) workshop in the spring of 2024.
The class of compact convex sets that arise as the state space of a \Cs{} is completely described by Alfsen--Hanche-Olsen--Schultz,  \cite[Corollary 8.6]{AH-OS:Acta}. 

The space of tracial states on a unital \Cs{} is always a Choquet simplex (if not empty). It follows that
the two Namioka--Phelps tensor products of the trace simplexes of two or more unital \Cs s coincide, and they are 
 shown to be the trace simplex of the (minimal or, equivalently, the maximal) tensor product of the \Cs s. As an application of this result we characterize when the trace simplex of the (minimal or maximal) tensor product of two or more \Cs s is a Bauer simplex. It was shown by Armstrong, \cite{Armstrong:Poulsen}, that the Poulsen simplex is prime, i.e.,  not a non-trivial tensor product; as a consequence, the trace simplex of a tensor product of two \Cs s is the Poulsen simplex only if one of the \Cs s has a unique trace and the trace simplex of the other is the Poulsen simplex. We include a new and elementary short proof of Armstrong's result.

The paper is organized as follows: In Section~\ref{sec:NP}, we review basic results about the Namioka--Phelps tensor products of convex compact sets, and we provide concrete examples of such tensor products. In Section~\ref{sec:states}, we analyze the minimal and the maximal tensor product of state spaces of \Cs s.  In Section~\ref{sec:traces}, we consider the tensor product of trace simplexes of  \Cs s. 

We thank Carlos Palazuelos for illuminating discussions on topics related to this paper, and Gabor Szabo, Mateusz Wasilewski,  Mizanur Rahaman, and Chi-Keung Ng for their comments to earlier versions of this paper. We also thank the anonymous referee for several valuable suggestions and for pointing out an incorrect statement about the Poulsen simplex in an earlier version of this paper. 

\section{Tensor products of compact convex sets} \label{sec:NP}

\noindent We review here some basic results about tensor products of convex compact sets, mostly taken from the Namioka--Phelps article, \cite{NP:PJM1969}. We provide  examples of tensor products of classical simplexes of relevance for this article. At the end  of the section we define infinite tensor products of simplexes, discuss their properties, and provide examples, as well.

Let $E$ be a real vector space, let $P$ be a \emph{proper generating cone} in $E$, i.e., $P+P \subseteq P$, $P \cap -P = \{0\}$, and $E = P-P$. An element $u \in P$ is an \emph{order unit} if for each $x \in E$ there exists $n \in \N$ such that $-nu \le x \le nu$ with respect to the ordering on $E$ arising from $P$: $x \le y$ if $y-x \in P$.  The state space $S(E,P,u)$ of such a triple consists of all linear maps $\varphi \colon E \to \R$ satisfying $\varphi(P) \subseteq \R^+$ and $\varphi(u)=1$. The space $S(E,P,u)$ is compact when equipped with  the weak$^*$-topology (the topology of pointwise convergence on $E$). Weak$^*$-compactness of $S(E,P,u)$ can fail when $u$ is not an order unit. 

We shall make use of the following Hahn--Banach type extension result, whose standard proof we omit: 

\begin{proposition} \label{prop:A} Let $(E,P,u)$ and $(E',P',u)$ be ordered real vector spaces with the same order unit $u$, where $E \subseteq E'$ and $P = P' \cap E$. Then each state on $(E,P,u)$ extends to a state on $(E',P',u)$.
\end{proposition}

\noindent
Let $K$ be a compact convex subset of a locally convex topological vector space $V$. Let $A(K)$ denote the ordered real vector space of continuous affine functions $f \colon K \to \R$,  let $A(K)^+$ denote the positive elements in $A(K)$, and let $u$ denote the constant function $1$, which is an order unit. Consider the pairing $\langle x, f \rangle = f(x)$, for $x \in K$ and $f \in A(K)$. Each $x \in K$ induces a state $\langle x, \, \cdot \, \rangle$ on $(A(K), A(K)^+, u)$. The map $x  \mapsto \langle x, \, \cdot \, \rangle$ is an affine homeomorphism from $K$ onto $S(A(K),A(K)^+,u)$, see \cite[Theorem 7.1]{Goo:partial}. 

Let $(E,P)$ be an ordered vector space, with positive cone $P$. The dual space $(E^*,P^*)$ consists of the usual algebraic dual $E^*$ of $E$, and the set $P^*$ of positive linear functionals on $E$. In the case where $E = A(K)$, and $K$ as above is a compact convex set, then $(A(K)^*)^+$ is the set of positive linear functionals on $A(K)$, each of which is proportional to a state on $A(K)$. Hence $\varphi \in (A(K)^+)^*$ if and only if $\varphi(f) = c \,  \langle x, f \rangle$, $f \in A(K)$, for  some $x \in K$, where $c = \varphi(1) \ge 0$. 

Let $(E_1, P_1,u_1), (E_2,P_2,u_2)$ be ordered real vector spaces with order units $u_1$ and $u_2$, respectively. Let $E_1 \otimes E_2$ denote their (algebraic) tensor product. Consider the subsets $P_1 \otimes P_2$ and $P_1 \tensorhat P_2$ of $E_1 \otimes E_2$, where the former is the cone spanned by elementary tensors $x_1 \otimes x_2$, with $x_j \in P_j$, and the latter is defined by 
\begin{eqnarray*}
P_1 \tensorhat P_2 
&=& \big\{x \in E_1 \otimes E_2 : \langle x, f_1 \otimes f_2 \rangle \ge 0, \, \text{for all} \, f_1 \in P^*_1, f_2 \in P^*_2\big \}.
\end{eqnarray*} 
It is clear that $P_1 \otimes P_2 \subseteq P_1 \tensorhat P_2$, and both satisfy the axioms of a generating positive cone. Moreover, $u_1 \otimes u_2$ is an order unit with respect to both positive cones. Indeed if $x_j \in E_j$ are so that $x_j \le \alpha_j u_j$, for $j=1,2$, then, upon expanding $(\alpha_1 u_1 -x_1) \otimes (\alpha_2 u_2 -x_2) \in P_1 \otimes P_2$, we see that 
$x_1 \otimes x_2 \le \alpha_1 \alpha_2 \, u_1 \otimes u_2$ with respect to $P_1 \otimes P_2$, hence with respect to $P_1 \tensorhat P_2$.

If $K_1$ and $K_2$ are compact convex sets, then
$$
A(K_1)^+ \tensorhat A(K_2)^+ = \big\{ f \in A(K_1) \otimes A(K_2) : \langle x_1 \otimes x_2, f \rangle \ge 0, \, \text{for all} \;  x_1 \in K_1, \, x_2 \in K_2\big\}.
$$
Let $K_1 \otimes^*K_2$ and $K_1 \otimes_* K_2$ denote the maximal, respectively, the minimal tensor products as defined in Namioka--Phelps, there denoted by $K_1 \square K_2$ and $K_1 \vartriangle K_2$, respectively, and defined as follows: 
 \begin{equation} \label{eq:min-tensor}
 K_1 \otimes^* K_2 = S\big(A(K_1) \otimes A(K_2), A(K_1)^+ \otimes A(K_2)^+, u_1 \otimes u_2\big),
 \end{equation}
 \begin{equation} \label{eq:max-tensor}
K_1 \otimes_* K_2 = S\big(A(K_1) \otimes A(K_2), A(K_1)^+ \tensorhat A(K_2)^+, u_1 \otimes u_2\big).
\end{equation}

\noindent
One has  $K_1 \otimes_* K_2 \subseteq K_1 \otimes^* K_2$, since $A(K_1)^+ \otimes A(K_2)^+ \subseteq A(K_1)^+ \tensorhat A(K_2)^+$. Both sets $K_1 \otimes_* K_2$ and $K_1 \otimes^* K_2$ are convex and  compact in the 
weak$^*$-topology (the latter because $u_1 \otimes u_2$ is an order unit in both cases). By definition of the 
weak$^*$-topology, a net $(z_i)$ in $K_1 \otimes_* K_2$ (or in $K_1 \otimes^* K_2$) converges to $z$ if and only if  $\langle z_i, f_1 \otimes  f_2 \rangle \to \langle z,f_1 \otimes f_2 \rangle$, for all $f_j \in A(K_j)$, $j=1,2$.

For $x_1 \in K_1$ and $x_2 \in K_2$ define a linear functional $\langle x_1 \otimes x_2, \, \cdot \, \rangle$ on $A(K_1) \otimes A(K_2)$ by 
\begin{equation} \label{eq:x1.x2}
\langle x_1 \otimes x_2, f_1 \otimes f_2 \rangle = \langle x_1, f_1 \rangle \langle x_2, f_2 \rangle,  \qquad f_j \in A(K_j).
\end{equation}
It is easy to verify that $\langle x_1 \otimes x_2, \, \cdot \, \rangle$ belongs to $K_1 \otimes_* K_2$, and we shall denote this functional simply by $x_1 \otimes x_2$. It gives rise to a bi-affine map 
\begin{equation} \label{eq:tensor}
 \otimes \colon K_1 \times K_2 \to K_1 \otimes_* K_2 \subseteq K_1 \otimes^* K_2,  \qquad (x_1, x_2) \mapsto
x_1 \otimes x_2, \qquad x_j \in K_j.
\end{equation}

In the opposite direction we have continuous affine surjective maps
\begin{equation} \label{eq:slice}
\pi_j \colon  K_1 \otimes^* K_2 \to K_j, \qquad j=1,2,
\end{equation}
defined by $\pi_j(x) = x_j$, $j=1,2$, where
\begin{equation} \label{eq:pi_j}
\langle x, f \otimes u_2 \rangle = \langle x_1, f \rangle, \qquad \langle x, u_1 \otimes g \rangle = \langle x_2, g \rangle,
\end{equation}
whenever $f \in A(K_1)$ and $g \in A(K_2)$. Clearly, $\pi_j(x_1 \otimes x_2) = x_j$, when $x_1 \in K_1$ and $x_2 \in K_2$. For the ``converse'', we have the following:

\begin{lemma}[Lemma 1.1 in \cite{NP:PJM1969}] \label{lm:slice}
Let $x \in K_1 \otimes^* K_2$ and suppose that one of $x_1=\pi_1(x)$ and $x_2=\pi_2(x)$ is an extreme point in $K_1$, respectively, $K_2$. Then $x = x_1 \otimes x_2$.
\end{lemma}

\begin{lemma} \label{lm:affine-indep} Let $K_1$ and $K_2$ be compact convex sets, and let $F_1 \subseteq K_1$ and $F_2 \subseteq K_2$ be affinely independent subsets. Then $F_1 \otimes F_2 :=\{ x_1 \otimes x_2 : x_j \in F_j\}$ is affinely independent in $K_1 \otimes_* K_2$. 
\end{lemma}

\begin{proof} A subset of a convex set is affinely independent if and only if all its finite subsets are affinely independent. We may therefore assume that $F_1$ and $F_2$ are finite, as well as affinely independent. Choose $f_x \in A(K_1)$ and $g_y \in A(K_2)$, for all $x \in F_1$ and $y \in F_2$, such that $\langle x', f_x \rangle = \delta_{x,x'}$ and $\langle y', g_y \rangle = \delta_{y,y'}$, for all $x,x' \in F_1$ and all $y,y' \in F_2$. Then
$$\langle x' \otimes y' , f_x \otimes g_y \rangle = \delta_{x,x'} \cdot \delta_{y,y'},$$
thus witnessing that $F_1 \otimes F_2$ is affinely independent.
\end{proof}

\noindent For a compact convex subset $K$ of a locally convex topological space $E$, let $\aff(K) \subseteq E$ denote the affine hull of $K$, and for each $r \ge 0$, set $\aff_r(K) = \{(r+1)x - ry : x,y \in K\}$. Note that $\aff_r(K)$ is a compact convex set containing $K$, and that $\aff(K) = \bigcup_{r \ge 0} \aff_r(K)$. We say that a subset $S \subseteq \aff(K)$ is \emph{bounded relatively to $K$} if $S \subseteq \aff_r(K)$, for some $r \ge 0$. 

The dimension, $\dim(K)$, of a finite dimensional convex  set $K$ is the number of elements in a maximal affinely independent subset of $K$ minus 1. It is also equal to the covering dimension of the relative interior, $K^{\mathrm{ri}}$, of $K$, and to the (linear) dimension of $\aff(K)$. 

We include the well-known result below for lack of a good reference (see also the comment below
Proposition~\ref{prop:tensor-I}).

\begin{proposition} \label{prop:dimK} Let $K_1$ and $K_2$ be finite dimensional compact convex sets. Then
$$\dim(K_1 \otimes^* K_2) = \dim(K_1 \otimes_* K_2) = (\dim(K_1)+1)(\dim(K_2)+1)-1.$$
Moreover, $K_1 \otimes^* K_2$ is bounded relatively to $K_1 \otimes_* K_2$.
\end{proposition}

\begin{proof} If $K$ is a finite dimensional compact convex set, then $A(K)$ has dimension equal to $\dim(K)+1$. Indeed, the linear map $A(K) \to \R^{\dim(K)+1}$ obtained by evaluating functions in $A(K)$ at any maximal affinely independent subset of $K$ is an isomorphism. 

Set $n = \dim(K_1)$ and $m= \dim(K_2)$. The real vector space $A(K_1) \otimes A(K_2)$ then has dimension $(n+1)(m+1)$. It follows that the dual space of $A(K_1) \otimes A(K_2)$ also has dimension $(n+1)(m+1)$. The hyperplane  consisting of $\varphi \in 
(A(K_1) \otimes A(K_2))^*$ satisfying $\varphi(u_1\otimes u_2)=1$ has dimension  $(n+1)(m+1)-1$, and $K_1 \otimes^* K_2$ is a convex subset of this hyperplane, and hence has dimension at most $(n+1)(m+1)-1$. Conversely, $K_1$ and $K_2$ contain affinely independent subsets $F_1$ and $F_2$ with $n+1$, respectively, $m+1$ elements. By Lemma~\ref{lm:affine-indep}, $F_1 \otimes F_2$ is an affinely independent subset of $K_1 \otimes_* K_2$, which therefore has dimension at least $(n+1)(m+1)-1$. 

As $K_1 \otimes_* K_2 \subseteq K_1 \otimes^* K_2$ and the two convex sets have the same dimension, we have 
$$K_1 \otimes^* K_2 \subseteq \aff(K_1 \otimes_* K_2) = \bigcup_{r \ge 0} \aff_r(K_1 \otimes_* K_2) = \bigcup_{r \ge 0} \aff_r(K_1 \otimes_* K_2)^{\mathrm{ri}},$$
where the last equality holds because the (relative) interior $\aff_{r+1}(K_1 \otimes_* K_2)^{\mathrm{ri}}$ contains $\aff_r(K_1 \otimes_* K_2)$.  Hence $K_1 \otimes^* K_2 \subseteq \aff_r(K_1 \otimes_* K_2)^\mathrm{ri} \subseteq \aff_r(K_1 \otimes _* K_2)$, for some $r \ge 0$, by compactness.
\end{proof}

\begin{remark} We show in Corollary~\ref{cor:unbounded} that $K_1 \otimes^* K_2$ fails to be bounded relatively to $K_1 \otimes_* K_2$, when $K_1$ and $K_2$ are the state spaces of (suitably non-commutative) \Cs s, and also that $K_1 \otimes^* K_2$ in general is not contained in $\aff(K_1 \otimes_* K_2)$.
\end{remark}

\noindent
The set of extreme points of a compact convex set $K$ is denoted by $\partial_e K$. 

\begin{proposition}[Namioka--Phelps, \cite{NP:PJM1969}] \label{prop:tensor-I} Let $K_1$ and $K_2$ be compact convex sets.
\begin{enumerate}
\item The bi-affine map $K_1 \times K_2 \to K_1 \otimes_* K_2$ is jointly continuous.
\item $\partial_e (K_1 \otimes_* K_2) = \{x_1 \otimes x_2 : x_1 \in \partial_e K_1,   x_2 \in \partial_e K_2\}$, and  the two sets $\partial_e K_1 \times \partial_e K_2$ and $\partial_e (K_1 \otimes_* K_2)$ are homeomorphic via (restriction of) the 
 map $K_1 \times K_2 \to K_1 \otimes_* K_2$.
\item  $\partial_e (K_1 \otimes_* K_2) \subseteq \partial_e (K_1 \otimes^* K_2)$.
\item $K_1 \otimes_* K_2 = \overline{\conv}\{x_1 \otimes x_2 : x_j \in K_j \} = 
\overline{\conv}\{x_1 \otimes x_2 : x_j \in \partial_e K_j\} $. 
\item The affine maps $\pi_j \colon K_1 \otimes^* K_2 \to K_j$ are continuous and surjective, and remain surjective when restricted to $K_1 \otimes_* K_2$.
\end{enumerate}
\end{proposition}

\noindent We shall sometimes rephrase (ii) by saying that $\partial_e (K_1 \otimes_* K_2) =  \partial_e K_1 \times \partial_e K_2$, identifying the right-hand set with 
$\{x_1 \otimes x_2 : x_1 \in \partial_e K_1,   x_2 \in \partial_e K_2\}$.

\begin{proof} Item (i) follows from \eqref{eq:x1.x2}. Items (ii) and (iii) are proved in  \cite[Theorem 2.3]{NP:PJM1969}. The part about the sets $\partial_e K_1 \times \partial_e K_2$ and $\partial_e (K_1 \otimes_* K_2)$ being homeomorphic follows from (i) and continuity of the maps in \eqref{eq:slice}.  Item (iv) follows from (iii). Finally, (v) follows from \eqref{eq:pi_j}.
\end{proof}

\noindent Namioka and Phelphs raise in  \cite{NP:PJM1969} the question if $K_1 \otimes_* K_2$ is always a face in $K_1 \otimes^* K_2$ (because of (iii) above). By Proposition~\ref{prop:dimK} in combination with (iii) above, if $K_1$ and $K_2$ are finite dimensional, then $K_1 \otimes_* K_2$ is  a face in $K_1 \otimes^* K_2$ only if the two sets are equal. This is likely to hold also in general.

\begin{theorem}[Namioka--Phelps, \cite{NP:PJM1969}] \label{thm:N-P}
Let $K$ be a compact convex set. Let $\square$ denote the square (in $\R^2$). The following are equivalent:
\begin{enumerate}
\item $K$ is a Choquet simplex,
\item $K \otimes_* K' = K \otimes^* K'$, for all compact convex sets $K'$,
\item $K \otimes_* \square = K \otimes^* \square$.
\end{enumerate}
\end{theorem}

\noindent It remains an open problem, credited to G.\ P. Barker, \cite{Barker:tensor} (there stated for finite dimensional cones), if one can replace the square in (iii) above with any other compact convex set which is not a Choquet simplex:

\begin{conjecture}[Barker] \label{conj:Barker}
 Let $K_1$ and $K_2$ be compact convex sets. Then $K_1 \otimes_* K_2 = K_1 \otimes^* K_2$ if and only if one of $K_1$ and $K_2$ is a Choquet simplex.
\end{conjecture}

\noindent Barker's conjecture was answered in the affirmative by Aubrun--Lami--Palazuelos--Plavala in \cite{ALPP:GAFA2021} in the case where both $K_1$ and $K_2$ are finite dimensional, see also \cite{Bruyn:cones}. 
The validity of Barker's conjecture in the infinite dimensional case was raised as open question by Aubrun at a BIRS (Granada) workshop in May 2024.

We shall henceforth let $K_1 \otimes K_2$ denote the (unique) tensor product of $K_1$ and $K_2$ when $K_1$ and $K_2$ are compact convex sets of which at least one is a Choquet simplex.

\begin{theorem}[Lazar, Davis--Vincent--Smith, Namioka--Phelps] \label{thm:simplex-tensor}
Let $K_1$ and $K_2$ be compact convex sets.
If $K_1$ and $K_2$ are Choquet simplexes, then so is $K_1 \otimes K_2$. Conversely, if one of $K_1 \otimes_* K_2$ and $K_1 \otimes^* K_2$ is a  Choquet simplex, then both $K_1$ and $K_2$ are  Choquet simplexes. 
\end{theorem}

\noindent The first statement of the theorem is due to Lazar, \cite{Lazar:simplex}, and Davis--Vincent-Smith, \cite{D-VS:simplex}, while the latter part is due to Namioka--Phelps, \cite[Proposition 2.10]{NP:PJM1969}. 

We say that a compact convex set is \emph{trivial} it is consists of one point.

\begin{proposition} \label{prop:elementary-tensors-dense}  Let $K_1$ and $K_2$ be compact convex sets.
The set $$\{x_1 \otimes x_2 : x_1 \in K_1, x_2 \in K_2\} \subseteq K_1 \otimes_* K_2$$ of elementary tensors in $K_1 \otimes_* K_2$ is closed and, if 
 both $K_1$ and $K_2$ are non-trivial, not equal to  the whole set $K_1 \otimes_* K_2$.
\end{proposition}

\begin{proof} Closedness follows from Proposition~\ref{prop:tensor-I} (i) and compactness of $K_1 \times K_2$. If $x,x' \in K_1$ and $y,y' \in K_2$ are such that $x \ne x'$ and $y \ne y'$, then it is routine to check that  $\frac12(x \otimes y + x' \otimes y')$ is not an elementary tensor.
\end{proof}

\noindent 
 Recall that the Poulsen simplex is the unique non-trivial  metrizable Choquet simplex with the property that its extreme boundary is dense. The result below was proved by Armstrong in \cite{Armstrong:Poulsen}. We present here a new elementary proof of this result. 

\begin{corollary}[Armstrong] \label{cor:not-Poulsen}
The Poulsen simplex is not a tensor product of two non-trivial compact convex sets.
\end{corollary}

\begin{proof} Let $K_1$ and $K_2$ be non-trivial compact convex sets. Each extreme point of $K_1 \otimes_* K_2$ is an elementary tensor by 
Proposition~\ref{prop:tensor-I} (ii), so
$$\overline{\partial_e (K_1 \otimes_* K_2)} \subseteq \{x_1 \otimes x_2 : x_j \in K_j\} \ne K_1 \otimes_* K_2,$$
by Proposition~\ref{prop:elementary-tensors-dense}. Hence $K_1 \otimes_* K_2$ cannot be the Poulsen simplex.
\end{proof}

\begin{example} \label{ex:tensor}
For each $n \ge 0$, let $\Delta_n$ denote the $n$-simplex, which is the Choquet simplex spanned by $n+1$ affinely independent points. In particular, $\Delta_0$ is the trivial simplex.
\begin{enumerate}
\item  $K \otimes \Delta_0 = K$, for all compact convex sets $K$.

\item $\Delta_{n-1} \otimes \Delta_{m-1} = \Delta_{nm-1}$, for all $n,m \ge 1$. 

\emph{Proof:} We know from Theorem~\ref{thm:simplex-tensor} that $\Delta_{n-1} \otimes \Delta_{m-1}$ is a Choquet simplex, and  $\partial_e(\Delta_{n-1} \otimes \Delta_{m-1})$ has $nm$ points by Proposition~\ref{prop:tensor-I} (ii).  \hfill $\square$

We can also deduce this without referring to Theorem~\ref{thm:simplex-tensor}, as follows: A finite dimensional compact convex set $\Delta$ is a simplex if and only if $\partial_e \Delta$ is affinely independent. Hence $\partial_e \Delta_{n-1}$ and $\partial_e \Delta_{m-1}$ are affinely independent sets with $n$, respectively, $m$ elements. It follows from Proposition~\ref{prop:tensor-I} that $\partial_e (\Delta_{n-1} \otimes \Delta_{m-1}) = \partial_e  \Delta_{n-1} \times \partial_e \Delta_{m-1}$, and Lemma~\ref{lm:affine-indep} tells us that this set is affinely independent, so $ \Delta_{n-1} \otimes  \Delta_{m-1}$ is a simplex of dimension $|\partial_e ( \Delta_{n-1} \otimes  \Delta_{m-1})| - 1 = nm-1$.

\item Let $S_1$ and $S_2$ be Choquet simplexes. Then $S_1 \otimes S_2$ is a Bauer simplex if and only if both $S_1$ and $S_2$ are Bauer simplexes, and $\partial_e (S_1 \otimes S_2)$ is homeomorphic to $\partial_e S_1 \times \partial_e S_2$. In particular, if $X$ and $X'$ are compact Hausdorff spaces, then $\mathrm{Prob}(X) \otimes \mathrm{Prob}(X') = \mathrm{Prob}(X \times X')$.

\emph{Proof:} By Proposition~\ref{prop:tensor-I}, the sets $\partial_e (S_1 \otimes S_2)$ and $\partial_e S_1 \times \partial_e S_2$ are homeomorphic, and $\partial_e S_1 \times \partial_e S_2$  is compact if and only if both $\partial_e S_1$ and $\partial_e S_2$ are compact. 
 \hfill $\square$

\end{enumerate}
\end{example}

\noindent {\bf{Infinite tensor products.}} For a sequence $(S_n)_{n \ge 1}$ of Choquet simplexes, define its infinite tensor product $\bigotimes_{k \ge 1} S_k$ to be the inverse limit
\begin{equation} \label{eq:inf-tensor}
\xymatrix{S_1  & S_1 \otimes S_2 \ar[l]_-{\pi_1} & S_1 \otimes S_2 \otimes S_3 \ar[l]_-{\pi_2} & \ar[l]_-{\pi_3} \cdots &
\bigotimes_{k \ge 1} S_k, \ar[l]}
\end{equation}
where $\pi_n \colon \big(\bigotimes_{k=1}^n S_k\big) \otimes S_{n+1} \to \bigotimes_{k=1}^n S_k$, $n \ge 1$, are the continuous affine surjections defined in \eqref{eq:slice} and \eqref{eq:pi_j}. By the definition of inverse limits, $\bigotimes_{k \ge 1} S_k$ is the set of all coherent sequences $(x_n)_{n \ge 1}$ in the product space $\prod_{n \ge 1} \bigotimes_{k=1}^n S_k$ (equipped with the product topology), i.e., $\pi_{n}(x_{n+1}) = x_n$, for all $n \ge 1$, making it a compact convex set, and moreover a simplex (as simplexes are closed under inverse limits). Denote by $\pi_{\infty,n}$ the canonical affine surjection $\bigotimes_{k \ge 1} S_k \to \bigotimes_{k=1}^n S_k$. 

For each $n \ge 1$, define also the continuous affine surjection $\pi'_n \colon \bigotimes_{k \ge 1} S_k \to S_n$ to be the composition of $\pi_{\infty,n}$ with the map $ \big(\bigotimes_{k=1}^{n-1} S_k\big) \otimes S_{n} \to S_n$ from \eqref{eq:slice} and \eqref{eq:pi_j}. Each sequence $(y_n)_{n \ge 1}$, with $y_n \in S_n$, defines an element $y= y_1 \otimes y_2 \otimes y_3 \otimes \cdots$ in  $\bigotimes_{k \ge 1} S_k$  satisfying $\pi_{\infty,n}(y) = y_1 \otimes y_2 \otimes \cdots \otimes y_n$, and $\pi'_n(y)=y_n$, for all $n \ge 1$. 

\begin{lemma} \label{lm:inftensor} For an element $x$ in $\bigotimes_{k \ge 1} S_k$, the following are equivalent:
\begin{enumerate}
\item $x \in \partial_e \big(\bigotimes_{k \ge 1} S_k\big)$,
\item $\pi_{\infty,n}(x) \in \partial_e \big(\bigotimes_{k=1}^n S_k\big)$, for all $n \ge 1$,
\item $x = x_1 \otimes x_2 \otimes x_3 \otimes \cdots$, for some $x_n \in \partial_e S_n$.
\end{enumerate}
\end{lemma}

\begin{proof} (ii) $\Rightarrow$ (i).  Suppose that (ii) holds and write $x = ty+(1-t)z$ with $y,z \in  \bigotimes_{k \ge 1} S_k$ and $0 < t < 1$. Then $\pi_{\infty,n}(x) = \pi_{\infty,n}(y) = 
\pi_{\infty,n}(z)$, for all $n \ge 1$, so $x=y=z$.

(i) $\Rightarrow$ (ii). Assume that (i) holds, and let $n \ge 1$. Set 
$K_1 = \bigotimes_{k=1}^n S_k$ and $K_2 = \bigotimes_{k \ge n+1} S_k$. By Proposition~\ref{prop:tensor-I}, each extreme point $x$ of $K_1 \otimes K_2$ is of the form $x = z_1 \otimes z_2$, with $z_j \in \partial_e K_j$.  We may identify $K_1 \otimes K_2$ with $\bigotimes_{k \ge 1} S_k$   and $z_1$ with $\pi_{\infty,n}(x)$, so (ii) holds.

(ii) $\Leftrightarrow$ (iii).  Set $x_n = \pi'_n(x)$ and $y = x_1 \otimes x_2 \otimes x_3 \otimes \cdots$. If (ii) holds, then $\pi_{\infty,n}(x) = x_1 \otimes x_2 \otimes \cdots \otimes x_n = \pi_{\infty,n}(y)$, for all $n \ge 1$, by Proposition~\ref{prop:tensor-I} and Lemma~\ref{lm:slice}, which implies that $x=y$. It also follows from Proposition~\ref{prop:tensor-I} that (iii) $\Rightarrow$ (ii). 
\end{proof}

\begin{example} \label{ex:inftensor} Let $S_1, S_2, S_3, \dots$ be a sequence of Choquet simplexes. Then $\bigotimes_{k \ge 1}S_k$ is a Bauer simplex if and only if each $S_k$ is a Bauer simplex, in which case $$\partial_e \Big( \bigotimes_{k \ge 1}S_k \Big)= \prod_{k \ge 1} \partial_e S_k,$$
when identifying the right-hand side with the set of elementary tensors $x_1 \otimes x_2 \otimes \cdots$, with $x_k \in \partial_e S_k$.

\medskip
\noindent
\emph{Proof:}  By Lemma~\ref{lm:inftensor}, $\partial_e \Big( \bigotimes_{k \ge 1}S_k \Big)$ is the inverse limit $\varprojlim \partial_e \big(\bigotimes_{k=1}^n S_k\big)$ and therefore compact if and only if each $\bigotimes_{k=1}^n S_k$ is a Bauer simplex. By Example~\ref{ex:tensor}, this happens if and only if each $S_k$ is a Bauer simplex. If this holds, then 
$$\partial_e \Big( \bigotimes_{k \ge 1}S_k \Big)= \varprojlim \partial_e \big(\bigotimes_{k=1}^n S_k\big) = \varprojlim  \prod_{k = 1}^n \partial_e S_k = \prod_{k \ge 1} \partial_e S_k. $$

\vspace{-.6cm}
\hfill $\square$ 
\end{example}

\noindent
The example shows  that any infinite tensor product, $\bigotimes_{k\ge 1} \Delta_{n_k}$, $n_k \ge 1$, of finite dimensional simplexes,  is the Bauer simplex $S$ whose extreme boundary is the Cantor set. 

If $P$ is the Poulsen simplex, then $P \otimes P \ne P$ by Corollary~\ref{cor:not-Poulsen}. Still, $P \otimes P$ has two important properties of the Poulsen simplex, namely each metrizable Choquet simplex is homeomorphic to a face in $P \otimes P$ (because $P$ is a face in $P \otimes P$ via $P \cong P \otimes \{x\}$, for any $x \in \partial_e P$) and the set of extreme points of $P \otimes P$ has no isolated points. This raises the question if the tensor powers $P^{\otimes n}$, for $1 \le n \le \infty$ eventually are equal or all mutually distinct.

It also seems interesting to explore more generally when a Choquet simplex $S$ is \emph{prime}, i.e., not a tensor product of two non-trivial compact convex sets (which then necessarily would have to be Choquet simplexes),  and, in the opposite direction, when $S \cong S^{\otimes \infty}$. Besides the Poulsen simplex (Corollary~\ref{cor:not-Poulsen}), also the finite dimensional simplexes $\Delta_{p-1}$ are prime whenever $p$ is a prime, cf.\ Example~\ref{ex:tensor} (ii) and Theorem~\ref{thm:simplex-tensor}. On the other hand, the Bauer simplex $S$ above with extreme boundary equal to the Cantor set is equal to its own infinite tensor power.

\section{Entanglement and tensor products of state spaces of \Cs s} \label{sec:states}

\noindent In this section, we identify the Namioka--Phelps tensor products of state spaces of two \Cs s. We show that Barker's conjecture holds for this class of (typically infinite dimensional) compact convex sets. The minimal tensor product of two state spaces is identified as the set of \emph{separable} (i.e., \emph{un-entangled}) states, and we also describe the maximal tensor product in terms of positive maps.  While  state spaces of \Cs s are special among general compact convex sets, it is an interesting and deep result, due to Alfsen--Hanche-Olsen--Schultz,  \cite[Corollary 8.6]{AH-OS:Acta}, that there is a complete axiomatic description of when a compact convex sets is the state space of a \Cs. 

It is an old and well-known fact that the state space of a \Cs{} is a Choquet simplex if and only if the \Cs{} is commutative, in which case the state space is a Bauer simplex. We reprove this fact in Theorem~\ref{thm:twoabelian} below, as a byproduct of our investigations. All Bauer simplexes arise in this way, as we can take the commutative \Cs{} to be the one whose spectrum is the extreme boundary of the given Bauer simplex.    

Let $\cA$ be a unital \Cs, let $S(\cA)$ denote the convex compact set of all states on $\cA$, and let $\cA_\sa$ denote the (real vector space) of all self-adjoint elements in $\cA$. The set of extreme points of $S(\cA)$ is by definition the set of pure states on $\cA$, which are those states $\rho$ for which the GNS representation $\pi_\rho$ is irreducible, i.e., $\pi_\rho(\cA)'' = B(H)$.

It is well-known (and a consequence of Kadison's representation theorem, see \cite[II.1.8]{Alf:convex}) that we can identify $A(S(\cA))$, the affine functions on $S(\cA)$, with $\cA_\sa$ via the pairing $\langle a, \rho \rangle = \rho(a)$, for $a \in 
\cA_\sa$ and $\rho \in S(\cA)$. This pairing further identifies $A(S(\cA))^+$ with $\cA^+$ and  the order unit $u=1 \in A(S(\cA))^+$  with $1_\cA$, and hence $(A(S(\cA)), A(S(\cA))^+, 1)$ with $(\cA_\sa, \cA^+, 1_\cA)$. This leads to the following familiar identification of the state space as $S(\cA) = S(\cA_\sa, \cA^+, 1_\cA)$.

Let now $\cA$ and $\cB$ be two unital \Cs s, and denote by $\cA \odot \cB$, $\cA \otimes \cB$, and
$\cA \otimes_\mx \cB$ their algebraic, minimal, and maximal tensor product, respectively. (We here deviate from the notation in Section~\ref{sec:NP} where $\otimes$ denoted the algebraic tensor product.) The algebraic tensor product $\cA_\sa \odot \cB_\sa$ is equal to $(\cA \odot \cB)_\sa$.  By the definition of the Namioka--Phelps tensor products $\otimes_*$ and $\otimes^*$ from Section~\ref{sec:NP} and the identifications above, we get:
\begin{eqnarray} \label{eq:tensor-states-1}
S(\cA) \otimes_* S(\cB) &=& S(( \cA \odot \cB)_\sa, \cA^+ \otensorhat \,  \cB^+ , 1_\cA \otimes 1_\cB), \\
\label{eq:tensor-states-2}
S(\cA) \otimes^* S(\cB) &=& S( (\cA \odot \cB)_\sa, \cA^+ \odot \cB^+ , 1_\cA \otimes 1_\cB).
\end{eqnarray}
Also, $\cA^+ \otensorhat  \cB^+$ consists of those $x \in (\cA \odot \cB)_\sa$ for which $(\rho_\cA \otimes \rho_\cB)(x) \ge 0$, for  all states $\rho_\cA \in S(\cA)$ and $\rho_\cB \in S(\cB)$. As remarked in the beginning of Section~\ref{sec:NP}, following the definitions of the positive cones $P_1 \otimes P_2$ and $P_1 \tensorhat P_2$,
$1_\cA \otimes 1_\cB$ is an order unit for $(\cA \odot \cB)_\sa$ with respect to both positive cones $\cA^+ \otensorhat \,  \cB^+$ and $\cA^+ \odot \cB^+$.

By taking restrictions to the self-adjoint part of the algebraic tensor product $\cA \odot \cB$, we may identify the state spaces $S(\cA \otimes \cB)$ and $S(\cA \otimes_\mx \cB)$  with the set of linear maps $(\cA \odot \cB)_\sa \to \R$ that extend (necessarily uniquely) to states on $\cA \otimes \cB$, respectively, $\cA \otimes_\mx \cB$. We may in this way view $S(\cA \otimes \cB)$ as a subset of $S(\cA \otimes_\mx \cB)$. Set 
\begin{equation}
S_*(\cA \otimes \cB) = \overline{\mathrm{conv}\{ \rho_\cA \otimes \rho_\cB : \rho_\cA \in S(\cA), \rho_\cB \in S(\cB) \}},
\end{equation}
which is a weak$^*$-closed convex subset of $S(\cA \otimes \cB)$. The states in $S_*(\cA \otimes \cB)$ are commonly referred to as being \emph{separable}. States in $S(\cA \otimes \cB)$ are \emph{entangled} if they are not separable. This agrees with the usual notion of entanglement, respectively, separability, in the finite dimensional case, cf.\ Remark~\ref{rem:entangled-matrices} below.

Given unital \Cs s $\cA$ and $\cB$, consider the following inclusions of  cones,
$$\cA^+ \odot \cB^+ \subseteq (\cA \odot \cB)^+, \quad \cA^+ \otimes \cB^+ \subseteq (\cA \otimes \cB)^+, \quad \cA^+ \otimes_\mx \cB^+ \subseteq (\cA \otimes_\mx \cB)^+,$$
where $\cA^+ \odot \cB^+$ is the (algebraic) cone generated by $a \otimes b$, with $a \in \cA^+$ and $b \in \cB^+$, and 
$\cA^+ \otimes \cB^+$ and $\cA^+ \otimes_\mx \cB^+$ are the closures of $\cA^+ \odot \cB^+$ with respect to the $C^*$-norms on $\cA \otimes \cB$ and $\cA \otimes_\mx \cB$, respectively. The positive cone $(\cA \odot \cB)^+$ is the cone generated by elements of the form $x^*x$, with $x \in \cA \odot \cB$. 

In analogy with the situation where $\cA$ and $\cB$ are matrix algebras, we will think of elements in  $\cA^+ \otimes \cB^+$ and $\cA^+ \otimes_\mx \cB^+$ as being separable, while the non-separable positive elements in  $(\cA \otimes \cB)^+$ and $(\cA \otimes_\mx \cB)^+$, respectively, are entangled. However, as pointed out in Remark~\ref{rem:aodotb},  the set difference $(\cA \odot \cB)^+ \setminus \cA^+ \odot \cB^+$ can be non-empty for reasons seemingly unrelated to entanglement.

\medskip \noindent The proposition below, which is a consequence of  the Namioka--Phelps characterization of the ``minimal'' tensor product $\otimes_*$  in Proposition~\ref{prop:tensor-I}, identifies the minimal tensor product of the state spaces of two \Cs s with the set of separable states:

\begin{proposition} \label{prop:statesAB} For unital \Cs s  $\cA$ and $\cB$ we have $S(\cA)   \otimes_*  S(\cB) = S_*(\cA \otimes \cB)$.
\end{proposition}

\begin{proof} Consider the affine continuous map
$$\gamma \colon S_*(\cA \otimes \cB) \to S(\cA) \otimes_* S(\cB)$$
obtained by restricting a state on $\cA \otimes \cB$ to $(\cA \odot \cB)_\sa$. The image of $\gamma$ is contained in the minimal tensor product $S(\cA) \otimes_* S(\cB)$,  because each state in $S_*(\cA \otimes \cB)$ is positive on $\cA^+ \otensorhat \cB^+$. The map $\gamma$ is injective because $(\cA \odot \cB)_\sa$ is dense in $(\cA \otimes \cB)_\sa$. 

We know from Proposition~\ref{prop:tensor-I} that $S(\cA)  \otimes_* S(\cB)$ is the closed convex hull of states in $S( (\cA \odot \cB)_\sa, \cA^+ \otensorhat  \cB^+ , 1_\cA \otimes 1_\cB)$ of the form $\rho_\cA \otimes \rho_\cB$, where $\rho_\cA \in S(\cA_\sa, \cA^+, 1_\cA) = S(\cA)$ and $\rho_\cB \in S(\cB_\sa, \cB^+, 1_\cB) = S(\cB)$. The functional $\rho_\cA \otimes \rho_\cB$ (defined on $(\cA \odot \cB)_\sa$) extends first to $\cA \odot \cB$, and next, 
by the definition of the minimal tensor product, to $\cA \otimes \cB$.  This shows that $\gamma$ is also surjective, and hence an affine homeomorphism.
\end{proof}

\noindent One has the following description of the states on the maximal tensor product, see also
\cite[Corollary 11]{LSSW:Schmidt}.

\begin{lemma} \label{lm:max-states}
For unital \Cs s  $\cA$ and $\cB$ we have 
$$S(\cA \otimes_\mx \cB) = S\big((\cA \odot \cB)_\sa, (\cA \odot \cB)^+, 1_\cA \otimes 1_\cB\big).$$
\end{lemma} 

 \begin{proof}
 We show that the  map $$\gamma' \colon S(\cA \otimes_\mx \cB) \to S((\cA \odot \cB)_\sa, (\cA \odot \cB)^+, 1_\cA \otimes 1_\cB)$$ obtained by restricting a state  on $\cA \otimes_\mx \cB$ to $(\cA \odot \cB)_\sa$,
 is an affine homeomorphism. It is clearly continuous, affine, and injective.  
 
Let $\rho \in S((\cA \odot \cB)_\sa, (\cA \odot \cB)^+, 1_\cA \otimes 1_\cB)$ and extend $\rho$ to a unital positive linear (complex valued) functional on $\cA \odot \cB$. Since $1_\cA \otimes 1_\cB$ is an order unit for $(\cA \odot \cB)_\sa$ with respect to the positive cone $\cA^+ \odot \cB^+$, it is also an order unit with respect to the larger positive cone $(\cA \odot \cB)^+$. Hence, for each $x \in \cA \odot \cB$, there is $R_x \ge 0$ such that $x^*x \le R_x^2 \, (1_\cA \otimes 1_\cB)$, i.e., $R_x^2 \, (1_\cA \otimes 1_\cB) - x^*x \in (\cA \odot \cB)^+$. It follows that $y^*x^*xy \le R_x^2 \, y^*y$, which implies that $\rho(y^*x^*xy) \le R_x^2 \,  \rho(y^*y)$, for all $y \in \cA \odot \cB$. Left-multiplication by $x$ on the pre-Hilbert space given by $\cA \odot \cB$ equipped with the inner product arising from $\rho$ is therefore bounded with bound $R_x$. 
This allows us to define the  GNS-representation $\pi_\rho$ of $\cA \odot \cB$.

By the definition of the maximal tensor product norm, $\|\pi_\rho(x)\| \le \|x\|_\mx$, for $x \in \cA \odot \cB$.
The state $\rho$ is continuous with respect to the norm $\|\pi_\rho( \cdot) \|$, and therefore also with respect to the norm $\| \, \cdot \, \|_\mx$. Hence $\rho$ extends to $\cA \otimes_\mx \cB$ which shows that $\gamma'$ is surjective.
\end{proof}

\noindent By Proposition~\ref{prop:statesAB} and Lemma~\ref{lm:max-states} above, and since $\cA^+ \odot \cB^+ \subseteq (\cA \odot \cB)^+ \subseteq \cA^+ \otensorhat \cB^+$, we obtain the following inclusions:
\begin{equation} \label{eq:inclusions}
S(\cA) \otimes_* S(\cB) = S_*(\cA \otimes \cB) \subseteq S(\cA \otimes \cB) \subseteq S(\cA \otimes_\mx \cB) \subseteq S(\cA) \otimes^* S(\cB).
\end{equation}
The first and last of the three inclusions above are strict when both $\cA$ and $\cB$ are non-commutative (see Theorem~\ref{thm:twoabelian} below). The inclusion $S(\cA \otimes \cB) \subseteq S(\cA \otimes_\mx \cB)$ is strict precisely when $\cA \otimes \cB \ne \cA \otimes_\mx \cB$.

\begin{remark}[Entangled states in finite dimensions] \label{rem:entangled-matrices}
When $\cA$ and $\cB$ are matrix algebras, one can identify all sets in \eqref{eq:inclusions} in terms of well-known sets of matrices in $\cA \otimes \cB$,  and whereby the set of separable states,  $S_*(\cA \otimes \cB) = S(\cA) \otimes_* S(\cB)$ on $\cA \otimes \cB$, corresponds to separable positive matrices, as explained below.

Recall that the vector space of $k \times k$ matrices is a Hilbert space with inner product  $\langle S,T \rangle = \Tr(T^*S)$, $S,T \in M_k(\C)$, for $k \ge 2$.  The Hilbert space structure gives a (conjugate linear) isomorphism from $M_k(\C)$ to its dual space $M_k(\C)^*$ given by $T \mapsto \langle \, \cdot \, , T \rangle$. The functional $\langle \, \cdot \, , T \rangle$ is unital if and only if $\Tr(T)=1$, is positive if and only if $T$ is positive, and
$\|\langle \, \cdot \, , T \rangle\| = \|T\|_1$, for $T \in M_k(\C)$.
For a (proper) cone $\mathcal{C} \subseteq M_k(\C)_\sa$, let $\mathcal{C}^\# \subseteq M_k(\C)_\sa$ denote its dual cone with respect to the inner product considered above. 

Let now $n,m \ge 2$ be integers. A positive element $T$ in $M_n(\C) \otimes M_m(\C)$, and its corresponding linear functional $\langle \, \cdot \, , T \rangle$, is separable if it  belongs to $M_n(\C)^+ \otimes M_m(\C)^+$, and entangled otherwise. We have the following duality of cones:
$$(M_n(\C)^+ \otimes M_m(\C)^+)^\# = M_n(\C)^+ \tensorhat M_m(\C)^+, $$
cf. \cite{3xH2001}, and $(M_n(\C) \otimes M_m(\C))^+$ is self-dual. By finite dimensionality, $M_n(\C) \odot M_m(\C) = M_n(\C) \otimes M_m(\C)$ and  $M_n(\C)^+ \odot M_m(\C)^+ = M_n(\C)^+ \otimes M_m(\C)^+$. Combining these facts with \eqref{eq:tensor-states-1} and \eqref{eq:tensor-states-2},  the convex sets
$$S(M_n(\C)) \otimes_* S(M_m(\C))  \subseteq S(M_n(\C) \otimes M_m(\C))  \subseteq S(M_n(\C)) \otimes^* S(M_m(\C)) $$
are identified (one by one) with the matrices of trace 1 in the sets
$$ M_n(\C)^+ \otimes M_m(\C)^+ \subseteq \big(M_n(\C) \otimes M_m(\C)\big)^+ \subseteq M_n(\C)^+ \tensorhat M_m(\C)^+,$$
after identifying functionals with their density matrices,
%
while $S(M_n(\C) \otimes M_m(\C)) = S(M_n(\C) \otimes_\mx M_m(\C))$, by nuclearity.
\end{remark}

\medskip
\noindent Let $\mathrm{Pos}(n,m)$ denote the set of positive linear maps $M_n(\C) \to M_m(\C)$, where $n,m \ge 2$, and let $\mathrm{UPos}(n,m)$ denote the set of unital maps in $\mathrm{Pos}(n,m)$. Let $\rho^{(m)}_0$ denote the maximally entangled state on $M_m(\C) \otimes M_m(\C)$ given by the unit vector $m^{-1/2} \sum_{j=1}^m e_j \otimes e_j$, where $(e_1, \dots, e_m)$ is an orthonormal basis for $\C^m$.  Observe that $\rho^{(m)}_0(A \otimes I_m)  = \rho^{(m)}_0(I_m \otimes A) = \tr(A)$, for all $A \in M_m(\C)$, where $\tr = m^{-1}\,\Tr$ is the normalized trace. 

The maximal tensor product of state spaces of matrix algebras has the following  description in terms of positive maps:

\begin{proposition} \label{prop:max-tensor-matrix}
Let $n,m \ge 2$. Then
\begin{eqnarray*}
S(M_n(\C)) \otimes^* S(M_m(\C)) &= &\big\{\rho^{(m)}_0 \circ (\Phi \otimes \id_m) : \Phi \in \mathrm{Pos}(n,m), \,   \tr(\Phi(I_n)) = 1\big\}\\
&=&\big\{\rho \circ (\Phi \otimes \id_m) : \Phi \in \mathrm{UPos}(n,m), \,  \rho \in S(M_m(\C) \otimes M_m(\C))  \big\}.
\end{eqnarray*}
\end{proposition}

\begin{proof} If $\Phi \colon M_n(\C) \to M_m(\C)$ is positive, then $\Phi \otimes \id_m$ maps $M_n(\C)^+ \otimes M_m(\C)^+$ into $M_m(\C)^+ \otimes M_m(\C)^+ \subseteq (M_m(\C) \otimes M_m(\C))^+$, so $\rho \circ (\Phi \otimes \id_m)$ is positive on $M_n(\C)^+ \otimes M_m(\C)^+$, for each state $\rho$ on $M_m(\C) \otimes M_m(\C)$. This shows that the two sets on the right-hand side are contained in $S(M_n(\C)) \otimes^* S(M_m(\C))$. 

For the converse direction,  let 

\vspace{-.2cm}
$$S^{(m)} = \sum_{i,j=1}^m E_{ij} \otimes E_{ji}, \qquad H^{(m)} = (\id_m \otimes t_m)(S^{(m)}) =  \sum_{i,j=1}^m E_{ij} \otimes E_{ij},$$ 

\vspace{-.2cm} \noindent
where $t_m$ is the transpose map on $M_m(\C)$, and $(E_{ij})$ is the set of matrix units for $M_m(\C)$. Then $\rho_0^{(m)} = m^{-1} \langle \, \cdot \, , H^{(m)}\rangle$. For a linear map $\Psi \colon M_m(\C) \to M_n(\C)$, consider its Jamiolkowski matrix $J(\Psi) = (\Psi \otimes \id_m)(S^{(m)})$ and its Choi matrix $C(\Psi) = (\Psi \otimes \id_m)(H^{(m)})$. It was shown by Jamiolkowski, \cite{Jam-1972}, that $\Psi$ is positive if and only if $J(\Psi) \in M_n(\C)^+ \tensorhat M_m(\C)^+$, which again is equivalent to $C(\Psi) = (\id_n \otimes t_m)(J(\Psi)) \in M_n(\C)^+ \tensorhat M_m(\C)^+$. 

Fix $\phi \in S(M_n(\C)) \otimes^* S(M_m(\C))$ and write  $\phi = \langle \, \cdot \, , T \rangle$, for some density matrix $T$ in $M_n(\C)^+ \tensorhat M_m(\C)^+$ with $\Tr(T) = 1$, cf.\ Remark~\ref{rem:entangled-matrices}. By Jamiolkowski's theorem and the  comments above, $T = C(\Psi)$ for some $\Psi \in \mathrm{Pos}(m,n)$. Hence, 
for $X \in M_n(\C) \otimes M_m(\C)$,
\begin{eqnarray*}
\phi(X)    &=  & \langle X, C(\Psi) \rangle = \langle X, (\Psi \otimes \id_m)(H^{(m)}) \rangle  \\
&=& \langle (\Psi^* \otimes \id_m)(X), H^{(m)} \rangle = m (\rho_0^{(m)} \circ (\Psi^* \otimes \id_m))(X).
\end{eqnarray*}
It follows that $\phi = \rho_0^{(m)} \circ (\Phi \otimes \id_m)$ with $\Phi =m \Psi^* \in \mathrm{Pos}(n,m)$, since positivity is preserved by taking adjoints. Also, $1 = \phi(I_n \otimes I_m) = \rho_0^{(m)}(\Phi(I_n) \otimes I_m) = \tr(\Phi(I_n))$.

For the  last claim,  we show that if $\Phi \in \mathrm{Pos}(n,m)$ with $\tr(\Phi(I_n))=1$, then there are $\Psi \in \mathrm{UPos}(n,m)$ and $\rho \in S(M_n(\C) \otimes M_n(\C))$ such that $\rho_0^{(m)} \circ (\Phi \otimes \id_m) = \rho \circ (\Psi \otimes \id_m)$. Let $P$ be the range projection of the positive element $\Phi(I_n)$ in $M_m(\C)$, and let $R$ be the inverse of $\Phi(I_n)^{1/2}$ relatively to $PM_m(\C)P$, i.e., $R \Phi(I_n)^{1/2} = P = \Phi(I_n)^{1/2}R$. Choose any unital positive map $\Psi_1 \colon M_n(\C) \to (I_m-P)M_m(\C)(I_m-P)$, e.g., $\Psi_1$ could be a state on $M_m(\C)$ followed by multiplication by $I_m-P$. Define
$$\Psi(A) = R\Phi(A)R + \Psi_1(A), \quad \rho(X) = \rho_0^{(m)}\big((\Phi(I_n)^{1/2} \otimes I_m)X (\Phi(I_n)^{1/2} \otimes I_m)\big),$$
for $A \in M_m(\C)$ and $X \in M_m(\C) \otimes M_m(\C)$. Then $\rho_0^{(m)} \circ (\Phi \otimes \id_m) =  \rho \circ (\Psi \otimes \id_m)$ and $\Psi \in \mathrm{UPos}(n,m)$. The functional $\rho$ is positive and $\rho(I_m \otimes I_m) = 
\rho_0^{(m)} \circ (\Phi \otimes \id_m)(I_n \otimes I_m)=1$.  Hence $\rho$ is a state on $M_m(\C) \otimes M_m(\C)$.
\end{proof}

\begin{remark} \label{rem:S} For later use we record the well-known fact that $S^{(m)}$ defined above belongs to $M_m(\C)^+ \tensorhat M_m(\C)^+$. This is a much used fact in quantum information theory, referred to by saying that $S^{(m)}$ is an \emph{entanglement witness}: If $T \in (M_m(\C) \otimes M_m(\C))^+$ with $\Tr(T) = 1$ is such that $\langle S^{(m), T} \rangle < 0$, then the state $\langle \, \cdot \, , T \rangle$ is entangled. 

One can see that $S^{(m)} \in M_m(\C)^+ \tensorhat M_m(\C)^+$ as follows: Consider the transpose map $t_m \colon M_m(\C) \to M_m(\C)$, which is positive, but not completely positive. The matrix $H^{(m)}$ defined in the proof of Proposition~\ref{prop:max-tensor-matrix} is positive.  For all states $\rho_1,\rho_2$ on $M_m(\C)$, $\rho_2 \circ t_m$ is a state on $M_m(\C)$, so $(\rho_1 \otimes \rho_2)(S^{(m)}) = (\rho_1 \otimes (\rho_2 \circ t_m))(H^{(m)}) \ge 0$.
\end{remark}

\begin{remark} \label{rem:max-positive} In general, if $\cA$ and $\cB$ are two unital \Cs s, and if $\Phi \colon \cA \to \cA_1$ and $\Psi \colon \cB \to \cB_1$ are unital positive linear maps into unital \Cs s $\cA_1$ and $\cB_1$, then $\Phi \otimes \Psi$ is a well-defined linear map $(\cA \odot \cB)_\sa \to (\cA_1 \odot \cB_1)_\sa$ which maps $\cA^+ \odot \cB^+$ into $\cA_1^+ \odot \cB_1^+ \subseteq (\cA_1 \odot \cB_1)^+$. (In general, $\Phi \otimes \Psi$ may be unbounded and will not extend to the completion $\cA \otimes_\mx \cB$.) Each state $\rho$ in $S(\cA_1 \otimes_\mx \cB_1)$ therefore defines an element $\rho \circ (\Phi \otimes \Psi)$ in $S(\cA) \otimes^* S(\cB)$. 

We do not know if all functionals in $S(\cA) \otimes^* S(\cB)$ arise in a way analogous to the one given in Proposition~\ref{prop:max-tensor-matrix}, as is the case when $\cA$ and $\cB$ are matrix algebras, or in the more general way described here. 
\end{remark}

\medskip \noindent We will show in Theorem~\ref{thm:twoabelian}  that entanglement always occurs in the tensor product of two non-commu\-ta\-tive \Cs s. This is well-known for matrix algebras, and also for general  \Cs s  in the case of entanglement of states. 
We choose here an approach towards this result, that we believe is of independent interest. Namely, entanglement phenomena, known to hold for matrix algebras, are  lifted to general non-commutative \Cs s
via Glimm's lemma, which  states that any non-commutative \Cs{} admits a sub-quotient isomorphic to $M_n(\C)$, for some $n \ge 2$. We then establish permanence results for entanglement with respect to subalgebras and quotients, cf.\ 
Propositions~\ref{prop:inclusion}, \ref{prop:entanglement-quotient} and \ref{prop:regular-positive} below. These results lead to Theorem~\ref{thm:twoabelian}.

\begin{definition} For a \Cs{} $\cA$, let $\rank(\cA) \in \N \cup \{\infty\}$ be the supremum of the set of all $n \in \N \cup \{\infty\}$ for which there exists an  irreducible representation of $\cA$ on a Hilbert space $H$ of dimension $n$. 
\end{definition}

\begin{remark} \label{rem:subhomogeneous} A \Cs{} $\cA$ is commutative if and only if $\rank(\cA)=1$. Moreover, a \Cs{}
has finite rank if and only if it is \emph{sub-homogeneous}, i.e., is a sub-\Cs{} of $M_n(C_0(X))$, for some locally compact Hausdorff space $X$ and some integer $n \ge 1$. Indeed,  $\mathrm{rank}(\cA) = \mathrm{rank}(\cA^{**})$, where $\cA^{**}$ is the bidual of $\cA$, which is a von Neumann algebra. If $ \mathrm{rank}(\cA^{**}) =n < \infty$, then $\cA^{**}$ is a  direct sum of type I$_k$ von Neumann algebras, with $1 \le k \le n$; a type I$_k$ von Neumann algebra is of the form $M_k(\cC)$, for some abelian von Neumann algebra $\cC$.
\end{remark}

\noindent For $n \ge 1$, let $CM_n  =\{f \in C([0,1], M_n): f(0) = 0\}$ denote the cone over $M_n=M_n(\C)$, and let $\cE_n$ denote the unitization of $CM_n$, which can be identified as $$\cE_n = \{f \in C([0,1], M_n): f(0) \in \C \cdot 1_n \}.$$ The lemma below is a variation of the well-known ``Glimm's lemma''. We include its proof for completeness of the exposition (and for lack of a precise reference).

\begin{lemma}[Glimm] \label{lm:Glimm} Let $\cA$ be a unital \Cs, and let $n \ge 1$ be an integer. The following conditions are equivalent:
\begin{enumerate}
\item $\rank(\cA) \ge n$,
\item there exists a non-zero \sh{} $CM_n \to \cA$,
\item there is an embedding of either $M_n$ or $CM_n$ into $\cA$,
\item there is a unital sub-\Cs{} $\cA_0$ of $\cA$ which has a quotient isomorphic to $M_n$. 
\end{enumerate} 
Items {\rm{(i)}}, {\rm{(ii)}} and {\rm{(iii)}} are equivalent also when $\cA$ is non-unital.
\end{lemma}

\begin{proof} (i) $\Rightarrow$ (ii). Suppose that $\rank(\cA) \ge n$ and let $\pi$ be an irreducible representation of $\cA$ on a Hilbert space $H$ of dimension at least $n$. Set $\cB = \pi(\cA)$. Let $P$ be a projection from $H$ onto an $n$-dimensional subspace $H_0$ of $H$. By Kadison's transitivity theorem, \cite[Theorem 5.4.3]{KadRin:VolI}, for each unitary $u$ on $H_0$ there is a unitary $v \in \cB$ whose restriction to $H_0$ is $u$. Necessarily, $v$ must commute with $P$, as we otherwise would have $\|v\|>1$.

Set 
$\mathcal{D} =\{a \in \cA : [\pi(a),P]=0\} \subseteq \cA$. Define $\varphi \colon \mathcal{D} \to B(H_0)$ by $\varphi(a) = \pi(a)|_{H_0}$. Then $\varphi$ is a unital \sh{} which, by the argument above, must be surjective. Identifying $B(H_0)$ with $M_n$ we get an isomorphism $M_n \to \mathcal{D}/\mathrm{ker}(\varphi)$, and hence, by composing with a surjection $CM_n \to M_n$, also a surjection $CM_n \to \mathcal{D}/\mathrm{ker}(\varphi)$. By projectivity of $CM_n$, this surjection lifts to a (non-zero) \sh{} $CM_n \to \mathcal{D}$.

(ii) $\Rightarrow$ (iii). Take a non-zero \sh{} $CM_n \to \cA$. The (non-zero) quotient of $CM_n$ by the kernel of this \sh{} is (isomorphic to) $\cB := \{f \in C(I,M_n) : f(0)=0\}$, for some (non-empty) closed subset $I$ of $[0,1]$. (The condition $f(0)=0$ is vacuous if $0 \notin I$.) By construction, $\cB$ embeds into $\cA$. 

If $I = [0,t]$ for some $0 < t \le 1$, then $\cB \cong CM_n$, and there is an embedding of $CM_n$ into $\cA$. If $I$ is not of this form, then $I$ contains (or is equal to) a non-empty clopen subset $I_0$ that does not contain $0$, and $C(I_0,M_n)$ is a direct summand of (or equal to) $\cB$. As $M_n$ embeds into $C(I_0,M_n)$, it embeds into $\cB$ and hence into $\cA$.

(iii) $\Rightarrow$ (iv). Suppose first that $CM_n$ embeds into $\cA$. Then its unitization $\cE_n$ admits a unital embedding into $\cA$, and $\cE_n$ has a quotient isomorphic to $M_n$, e.g., via point evaluation at $t=1$. If, instead, $M_n$ embeds into $\cA$, then we can take $\cA_0 = M_n$ if this embedding is unital. If not, then $\cA_0 = M_n \oplus \C$ embeds unitally into $\cA$.

(iv) $\Rightarrow$ (ii). Suppose that $\cA_0$ is a (unital) sub-\Cs{} of $\cA$ with a quotient isomorphic to $M_n$. Then $\cA_0$ has an irreducible representation of dimension $n$, so $\rank(\cA_0) \ge n$. By (i) $\Rightarrow$ (ii), this implies that there is a non-zero \sh{} $CM_n \to \cA_0 \subseteq \cA$.

(iii) $\Rightarrow$ (i) is clear. 
\end{proof}

\noindent The unital \Cs{} $\cA_0$ in (iv) above can be take to be one of the \Cs s $\cE_n$, $M_n$, or $M_n \oplus \C$. 

We shall now address the question if entangled states, or entangled positive elements, can become un-entangled  when passing to larger \Cs s or lifting from quotients. It is easy to see that the answer is no to the former for entangled states:

\begin{proposition} \label{prop:inclusion} For inclusions of unital \Cs s $\cA_0 \subseteq \cA$ and $\cB_0 \subseteq \cB$ we have the following commutative diagram, where the horizontal maps, $r$ and $r_*$, given by restriction, are surjective:
$$
\xymatrix{S_*(\cA \otimes \cB) \ar@{->>}[r]^-{r_*} \ar@{}[d]|-*[@]{\subseteq} & S_*(\cA_0 \otimes \cB_0) \ar@{}[d]|-*[@]{\subseteq} \\
S(\cA \otimes \cB) \ar@{->>}[r]^-r & S(\cA_0 \otimes \cB_0)}
$$
In particular, if $\rho_0 \in S(\cA_0 \otimes \cB_0)$ is entangled, then so is any extension  of $\rho_0$ to $\cA \otimes \cB$.
\end{proposition}

\begin{proof} The restriction mapping $r \colon S(\cA \otimes \cB) \to S(\cA_0 \otimes \cB_0)$ is  surjective (by Hahn--Banach), and it maps a product state $\rho_\cA \otimes \rho_\cB$, with $\rho_\cA \in S(\cA)$ and $\rho_\cB \in S(\cB)$, to the product state $\rho_{\cA_0} \otimes \rho_{\cB_0}$, where $\rho_{\cA_0}$ and $\rho_{\cB_0}$ are the restrictions of $\rho_\cA$ and $\rho_\cB$ to $\cA_0$ and $\cB_0$, respectively. Being affine and continuous, it follows that $r$ maps $S_*(\cA \otimes \cB)$ into $S_*(\cA_0 \otimes \cB_0)$. Conversely, any states $\rho_{\cA_0} \in S(\cA_0)$ and $\rho_{\cB_0} \in S(\cB_0)$ can be lifted to states $\rho_\cA \in S(\cA)$ and $\rho_\cB \in S(\cB)$, and $r_*(\rho_\cA \otimes \rho_\cB) = \rho_{\cA_0} \otimes \rho_{\cB_0}$, which shows that also $r_*$ is surjective. 
\end{proof}

\noindent We next prove a permanence result for entanglement of states with respect to quotients., cf.\ 
Proposition~\ref{prop:entanglement-quotient}. First we need two (standard) lemmas.

\begin{lemma} \label{lm:Milman}
Let $K$ be a compact convex set, let $F$ be a closed face in $K$, and let $X$ be a compact subset of $K$. Then
$$ \overline{\conv(X)} \cap F = \overline{\conv(X \cap F)}.$$
\end{lemma}

\begin{proof} By Milman's partial converse to the Krein--Milman theorem and compactness of $X$, we have $\partial_e \, (\overline{\conv(X)}) \subseteq X$, and since $F$ is a face in $K$ we have 
$$\partial_e \, (\overline{\conv(X)} \cap F) = \partial_e \, (\overline{\conv(X)}) \cap F \subseteq X \cap F.$$
It now follows from Krein--Milman that $ \overline{\conv(X)} \cap F \subseteq \overline{\conv(X \cap F)}$. The other inclusion is obvious.
\end{proof}

\begin{lemma} \label{lm:face}
Let $\cA$ and $\cA_1$ be unital \Cs s and let $\pi \colon \cA \to \cA_1$ be a (unital) surjective \sh. It follows that $\{\rho \circ \pi : \rho \in S(\cA_1)\}$ is a closed face in $S(\cA)$.
\end{lemma}

\begin{proof} Set $F = \{\rho \circ \pi : \rho \in S(\cA_1)\} \subseteq S(\cA)$. Since $F$ is the image of the compact convex set $S(\cA_1)$ under the continuous affine map $\rho \mapsto \rho \circ \pi$, we see that $F$ is compact and convex. Let $J$ denote the kernel of $\pi$, which is a closed two-sided ideal in $\cA$. A state  on $\cA$ belongs to $F$ if and only if it vanishes on $J$. Suppose that $\phi \in F$ is a convex combination $\phi = \alpha  \, \phi_1 + (1-\alpha) \, \phi_2$ of two states $\phi_1,\phi_2$ on $\cA$ with $0 < \alpha < 1$. Then $\phi_1$ and $\phi_2$ must vanish on all positive elements of $J$, and hence on all of $J$, so $\phi_1$ and $\phi_2$ both belong to $F$. This proves that $F$ is a face.
\end{proof}

\begin{proposition} \label{prop:entanglement-quotient} Let $\cA$, $\cA_1$, $\cB$, $\cB_1$ be unital \Cs s and let $\pi_\cA \colon \cA \to \cA_1$ and $\pi_\cB \colon \cB \to \cB_1$ be (unital) surjective \sh s. Let $\rho \in S(\cA_1 \otimes \cB_1)$. Then $\rho$ is entangled in $\cA_1 \otimes \cB_1$ if and only if  $\rho \circ (\pi_\cA \otimes \pi_\cB)$ is entangled in $\cA \otimes \cB$.
\end{proposition}

\begin{proof} Consider the following sets of states on $\cA \otimes \cB$:
$$F = \{ \rho \circ (\pi_\cA \otimes \pi_\cB) : \rho \in S(\cA_1 \otimes \cB_1)\}, \quad 
F_* = \{ \rho \circ (\pi_\cA \otimes \pi_\cB) : \rho \in S_*(\cA_1 \otimes \cB_1)\},$$
$$X = \{ \rho_\cA \otimes \rho_\cB : \rho_\cA \in S(\cA), \rho_\cB \in S(\cB)\}.$$
Then $X$ is compact (being the continuous image of the compact set $S(\cA) \times S(\cB)$), and $\overline{\conv(X)} = S_*(\cA \otimes \cB)$. 
We must therefore show that $\overline{\conv(X)} \cap F = F_*$. The inclusion $F_* \subseteq \overline{\conv(X)} \cap F$ is clear. To prove the reverse inclusion, it suffices to show that $X \cap F \subseteq F_*$ by Lemmas~\ref{lm:Milman} and \ref{lm:face}.
Take $\phi \in X \cap F$. Then 
$$\rho(\pi_\cA(a) \otimes \pi_\cB(b)) = \phi(a \otimes b) = \phi(a \otimes 1_\cB)\phi(1_\cA \otimes b) = \rho(\pi_\cA(a) \otimes 1_{\cB_1})\rho(1_{\cA_1} \otimes \pi_\cB(b)),$$
for all $a \in \cA$ and $b \in \cB$, since $\phi$ is a product state. It follows that $\rho$ is a product state, so $\phi \in F_*$.
\end{proof}

\noindent We pursue the theme of permanence properties in Proposition~\ref{prop:regular-positive} below,  where we show that entanglement of positive elements in a tensor product cannot un-entangle when passing to larger \Cs s, at least in the presence of some (mild) injectivity assumptions:

\begin{definition} \label{def:injective} A \Cs{} $\cA$ is said to be \emph{injective}, respectively, \emph{approximately injective} if for all inclusions $\cA \subseteq \cB$, there exists a completely positive contractive (cpc) map $E \colon \cB \to \cA$, respectively, a net $E_\alpha \colon \cB \to \cA$ of cpc maps, such that 
$E(a) = a$, respectively, $\lim_\alpha E_\alpha(a) = a$, for all $a \in \cA$. 

A map $E$ satisfying the conditions above is a conditional expectation, and a we call the net $(E_\alpha)$ an \emph{approximate conditional expectation}. 
\end{definition}

\noindent A (formally) weaker notion of approximate injectivity was considered by Effros--Haagerup in \cite{EffHaa:lift}. They show that (their version of) approximate injectivity implies the WEP of Lance; hence so does our version of approximate injectivity. Kirchberg proved in \cite{Kir:CEP} that the equivalence of the Effros--Haagerup approximate injectivity (which is condition (A4) in Kirchberg's paper) is equivalent to WEP if and only if the Connes Embedding Problem has an affirmative answer, which by \cite{JNVWY:MIP*=RE} is not the case!

Injectivity for \Cs s is quite rare (all finite dimensional \Cs s have this property), but approximate injectivity is much more common. For example, all nuclear \Cs s are approximately injective. Indeed, suppose that $\cA$ is a nuclear \Cs{} witnessed by a net $(\varphi_\alpha, M_{n(\alpha)}(\C), \psi_\alpha)$ where $\varphi_\alpha \colon \cA \to M_{n(\alpha)}(\C)$ and $\psi_\alpha \colon M_{n(\alpha)}(\C) \to \cA$ are cpc maps satisfying $\lim_\alpha (\psi_\alpha \circ \varphi_\alpha)(a) = a$, for all $a \in \cA$, and that $\cA \subseteq \cB$. By Arveson's extension theorem we may extend $\varphi_\alpha$ to a cpc map $\overline{\varphi}_\alpha \colon \cB \to M_{n(\alpha)}(\C)$. Then $E_\alpha = \psi_\alpha \circ \overline{\varphi}_\alpha \colon \cB \to \cA$ is an approximate conditional expectation.

\begin{proposition} \label{prop:regular-positive}
Let $\cA_0 \subseteq \cA$ and $\cB_0 \subseteq \cB$ be  \Cs s. Then {\rm{(i)}}, {\rm{(ii)}}, and {\rm{(iii)}} below hold if $\cA_0$ and $\cB_0$ are injective (e.g., finite dimensional); and {\rm{(ii)}} and {\rm{(iii)}} hold if $\cA_0$ and $\cB_0$ are approximately injective (e.g., nuclear):
\begin{enumerate}
\item $(\cA^+ \odot \cB^+) \cap (\cA_0 \odot \cB_0) = \cA_0^+ \odot \cB_0^+$,
\item $(\cA^+ \otimes \cB^+) \cap (\cA_0 \otimes \cB_0) = \cA_0^+ \otimes \cB_0^+$,
\item $\cA_0 \otimes_\mx \cB_0 \subseteq \cA \otimes_\mx \cB$ and $(\cA^+ \otimes_\mx \cB^+) \cap (\cA_0 \otimes_\mx \cB_0) = \cA_0^+ \otimes_\mx \cB_0^+.$
\end{enumerate}
\end{proposition}

\begin{proof} The inclusions ``$\supseteq$'' in (i)--(ii) hold trivially, for all $\cA_0$ and $\cB_0$, and likewise for the the analogous identity in (iii) when the first inclusion in (iii) holds.

 Suppose  that $\cA_0$ and $\cB_0$ are injective, and let $E_\cA \colon \cA \to \cA_0$ and $E_\cB \colon \cB \to \cB_0$ be conditional expectations. The (algebraic) tensor product map $E_\cA \odot E_\cB \colon \cA \odot \cB \to \cA_0 \odot \cB_0$ maps $\cA^+ \odot \cB^+$ into $\cA_0^+ \odot \cB_0^+$ and restricts to the identity on $\cA_0 \odot \cB_0$. Hence, if $x \in (\cA^+ \odot \cB^+) \cap (\cA_0 \odot \cB_0)$, then $x = (E_\cA \odot E_\cB)(x) \in \cA_0^+ \odot \cB_0^+$. 

Suppose next that $\cA_0$ and $\cB_0$ are approximately injective, and let $E^\alpha_\cA \colon \cA \to \cA_0$ and $E^\alpha_\cB \colon \cB \to \cB_0$ be approximate conditional expectations. Then
$$E^\alpha_\cA \otimes E^\alpha_\cB \colon \cA \otimes \cB \to \cA_0 \otimes \cB_0, \qquad 
E^\alpha_\cA \otimes_\mx E^\alpha_\cB \colon \cA \otimes_\mx \cB \to \cA_0 \otimes_\mx \cB_0$$
are approximate conditional expectations. 

Let $\iota_\cA \colon \cA_0 \to \cA$ and $\iota_\cB \colon \cB_0 \to \cB$ be the inclusion mappings. Then
$$\lim_\alpha (E^\alpha_\cA \otimes_\mx E^\alpha_\cB) \circ (\iota_\cA \otimes_\mx \iota_\cB)(x) = x,$$
 for all $x \in \cA_0 \otimes_\mx \cB_0$. Hence $\iota_\cA \otimes_\mx \iota_\cB$ is injective. 

The cpc maps $E^\alpha_\cA \otimes E^\alpha_\cB$  and $E^\alpha_\cA \otimes_\mx E^\alpha_\cB$ map $\cA^+ \odot \cB^+$ into  $\cA_0^+ \odot \cB_0^+$, and hence, by continuity, $\cA^+ \otimes \cB^+$ into $\cA_0^+ \otimes \cB_0^+$, respectively,  $\cA^+ \otimes_\mx \cB^+$ into $\cA_0^+ \otimes_\mx \cB_0^+$.

Let $x \in (\cA^+ \otimes \cB^+) \cap (\cA_0 \otimes \cB_0)$. Then $x = \lim_\alpha (E^\alpha_\cA \otimes E^\alpha_\cB)(x)$ and $(E^\alpha_\cA \otimes E^\alpha_\cB)(x)$ belongs to $\cA_0^+ \otimes \cB_0^+$, for all $\alpha$, so $x \in \cA_0^+ \otimes \cB_0^+$; likewise for $x \in (\cA^+ \otimes_\mx \cB^+) \cap (\cA_0 \otimes_\mx \cB_0)$.
\end{proof}

\noindent We have no examples of inclusions $\cA_0 \subseteq \cA$ and $\cB_0 \subseteq \cB$  of \Cs s where (i) or (ii) above fails. Note, however, that the first inclusion in (iii) does not hold in general, cf.\ \cite[Proposition 3.6.9]{BO-Approx}.

\begin{definition} \label{def:kappa}
Let $\cA$ and $\cB$ be unital \Cs s. For a linear functional $\rho$ on $\cA \odot \cB$ set $\|\rho\|_\mx = \sup \{ |\rho(x)| : x \in \cA \odot \cB, \, \|x\|_\mx \le 1\}$, and set
$$\kappa(\cA,\cB) = \sup \{ \|\rho\|_\mx : \rho \in S(\cA) \otimes^* S(\cB) \}.$$
\end{definition}

\noindent  Observe that $S(\cA \otimes_\mx \cB) = S(\cA) \otimes^* S(\cB)$ if and only if $\kappa(\cA,\cB) = 1$, which by the proposition below 
(and Theorem~\ref{thm:twoabelian}) happens  if and only if one of $\cA$ and $\cB$ is commutative.

\begin{proposition} \label{prop:kappa} Let $\cA$ and $\cB$ be unital \Cs s.
\begin{enumerate}
\item $n \le \kappa(M_n(\C), M_n(\C)) < \infty$, for each integer $n \ge 1$.
\item Let $\cA$ and $\cB$ be unital \Cs s both of rank at least $n \ge 2$. Then $\kappa(\cA,\cB) \ge n$.
\item If neither $\cA$ nor $\cB$ is sub-homogeneous, then $\kappa(\cA,\cB)=\infty$.
\item If $S(\cA) \otimes^* S(\cB) \subseteq \aff_r(S(\cA) \otimes_* S(\cB))$, i.e., $S(\cA) \otimes^* S(\cB)$ is bounded by $S(\cA) \otimes_* S(\cB)$ with a constant $r \ge 0$, then $\kappa(\cA,\cB) \le 2r+1$. 
\end{enumerate}
\end{proposition}

\begin{proof} (i). Let $S^{(n)} \in M_n(\C)^+ \otensorhat M_n(\C)^+$ be as in Remark~\ref{rem:S} and 
set $\rho = \langle \, \cdot \, , n^{-1} S^{(n)} \rangle$. 
Then $\rho \in S(M_n(\C)) \otimes^* S(M_n(\C))$, cf.\ 
Remark~\ref{rem:entangled-matrices}, since $\Tr(n^{-1} S^{(n)})= 1$; and $\|\rho\|_\mx = \|\rho\| = \|n^{-1} S^{(n)}\|_1 = n$.
Conversely, each $\rho$ in $S(M_n(\C)) \otimes^* S(M_n(\C))$ is a linear functional on the finite dimensional vector space $(M_n(\C) \otimes M_n(\C))_\sa$, and is therefore bounded.

(ii). 
By Lemma~\ref{lm:Glimm} there are (nuclear) unital sub-\Cs s $\cA_0 \subseteq \cA$ and $\cB_0 \subseteq \cB$ and unital surjections $\pi_1 \colon \cA_0 \to M_n(\C)$ and $\pi_2 \colon \cB_0 \to M_n(\C)$. This gives us a unital surjection 
$\pi_1 \otimes \pi_2 \colon  \cA_0 \odot \cB_0 \to M_n(\C) \otimes M_n(\C)$. 

By (i), there is $\rho$ in $S(M_n(\C)) \otimes^* S(M_n(\C))$ with $n \le \|\rho\|_\mx < \infty$. 
The functional $\rho \circ (\pi_1 \otimes \pi_2)$ on $\cA_0 \odot \cB_0$, then belongs to $S(\cA_0) \otimes^* S(\cB_0)$ and satisfies $n \le \|\rho \circ (\pi_1 \otimes \pi_2)\|_\mx < \infty$. We may therefore extend $\rho \circ (\pi_1 \otimes \pi_2)$ to the completion $\cA_0 \otimes_\mx \cB_0$ to obtain a state on $(\cA_0 \otimes_\mx \cB_0, \cA_0^+ \otimes_\mx \cB_0^+, 1_\cA \otimes 1_\cB)$.  By Proposition~\ref{prop:regular-positive} (iii) and Proposition~\ref{prop:A} this state can further  be extended to a state $\rho'$ on $(\cA \otimes_\mx \cB, \cA^+ \otimes_\mx \cB^+, 1_\cA \otimes 1_\cB)$. The restriction of $\rho'$ to $(\cA \odot \cB)_\sa$ yields a state $\rho''$ in $S(\cA) \otimes^* S(\cB)$ satisfying $\|\rho''\|_\mx \ge \|\rho \circ (\pi_1 \otimes \pi_2)\|_\mx \ge n$, thus giving the desired conclusion. 

(iii) is an immediate consequence of (ii) and Remark~\ref{rem:subhomogeneous}. 

(iv). If $\rho \in \aff_r(S(\cA) \otimes_* S(\cB))$, for some $r \ge 0$, then $\rho = (r+1)\rho_1 - r \rho_2$, for some $\rho_1,\rho_2 \in S(\cA) \otimes_* S(\cB) \subseteq S(\cA \otimes_\mx \cB)$, so $$\|\rho\|_\mx \le (r+1) \|\rho_1\|_\mx + r \|\rho_2\|_\mx \le 2r+1.$$ 
This proves (iv).
\end{proof}

\noindent We have the following description of $\kappa(\cA,\cB)$ in terms of unital positive maps in the case where $\cA$ and $\cB$ are matrix algebras:

\begin{theorem} \label{prop:kappa(n,m)} Let $n,m \ge 1$ be integers. Then
$$\kappa(M_n(\C), M_m(\C)) = \sup \{\|\Phi\|_{\mathrm{cb}} : \Phi \in \mathrm{UPos}(n,m)\} = \min\{n,m\}.$$
\end{theorem}

\begin{proof} 
It was shown in Proposition~\ref{prop:max-tensor-matrix} that each state in $S(M_n(\C)) \otimes^* S(M_m(\C))$ is of the form $\rho \circ (\Phi \otimes \id_m)$, for some $\Phi \in \mathrm{UPos}(n,m)$ and some state $\rho$ on $M_n(\C) \otimes M_m(\C)$. By  \cite[Proposition 8.11]{Paulsen:cb-book} (Smith) we have
$$\|\rho \circ (\Phi \otimes \id_m)\|_\mx \le \|\Phi \otimes \id_m\| =  \|\Phi\|_{\mathrm{cb}},$$
which shows that 
$$\kappa(M_n(\C), M_m(\C)) \le \sup \{\|\Phi\|_{\mathrm{cb}} : \Phi \in \mathrm{UPos}(n,m)\}.$$

From the recent paper, \cite{ADMPR:cb-norms}, we obtain that 
$$ \sup \{\|\Phi\|_{\mathrm{cb}} : \Phi \in \mathrm{UPos}(n,m)\} \le \min\{n,m\}.$$
Indeed, the ``$\le n$''-part is \cite[Theorem 3.7]{ADMPR:cb-norms}, and the ``$\le m$''-part is well-known (see p.\ 2 of the same paper), and is a consequence of the estimate $\|\Phi\|_{\mathrm{cb}} = \|\Phi \otimes \id_m\|  \le m \|\Phi\|$, due to Tomiyama, cf.\ \cite[Exercise 3.10]{Paulsen:cb-book}, and the fact $\|\Phi\| =1$, as $\Phi$ is unital and positive. 

Finally, it was shown in Proposition~\ref{prop:kappa} (ii) that $\min\{n,m\} \le \kappa(M_n(\C), M_m(\C))$.
%
%
\end{proof}

\noindent It seems likely that the result of Theorem~\ref{prop:kappa(n,m)} above extends to general \Cs s, and we expect that
\begin{equation} \label{eq:kappa?}
\kappa(\cA,\cB) = \min\{\rank(\cA), \rank(\cB)\},
\end{equation}
for all unital \Cs s $\cA$ and $\cB$.
The inequality $\ge$ follows from Proposition~\ref{prop:kappa} (ii).\footnote{Very recently, and after this paper was made available online, equality  in \eqref{eq:kappa?} was established in \cite{Rahaman-Wasilewski-Seprable}.}

We saw in Proposition~\ref{prop:dimK} that $K_1 \otimes^* K_2$ is bounded relatively to $K_1 \otimes_* K_2$, when $K_1$ and $K_2$ are finite dimensional compact convex sets. This does not hold  in general when $K_1$ and $K_2$ are infinite dimensional:

\begin{corollary} \label{cor:unbounded} Let $\cA$ and $\cB$ be unital \Cs s. 
\begin{enumerate}
\item If neither $\cA$ nor $\cB$ is sub-homogeneous, then $S(\cA) \otimes^* S(\cB)$ is not bounded relatively to $S(\cA) \otimes_* S(\cB)$.
\item If $\cA \otimes \cB \ne \cA \otimes_\mx \cB$, then $S(\cA) \otimes^* S(\cB) \nsubseteq \aff(S(\cA) \otimes_* S(\cB))$. 
\end{enumerate}
\end{corollary}

\begin{proof} Item (i) is an immediate consequence of Proposition~\ref{prop:kappa} (iii) and (iv).

(ii). Let $J$ be the (non-zero) kernel of the canonical surjection $\cA \otimes_\mx \cB \to \cA \otimes \cB$. We can then describe $S(\cA \otimes \cB)$ as the set of states $\rho$ on $\cA \otimes_\mx \cB$ that vanish on $J$. It follows that any functional in $\aff(S(\cA \otimes \cB))$ vanishes on $J$, so $S(\cA \otimes_\mx \cB) \nsubseteq \aff(S(\cA \otimes \cB))$. 
\end{proof}

\noindent We summarize our discussion about tensor product of state spaces of \Cs s in the following theorem.
Note in particular that the theorem implies that $S(\cA) \otimes_* S(\cB) = S(\cA) \otimes^* S(\cB)$ if and only if one of $S(\cA)$ and $S(\cB)$ is a Choquet simplex, thus confirming Barker's conjecture for compact convex sets arising as the state space of \Cs s. The theorem states  that entanglement (both of states and of positive elements) always exists in the presence of non-commutativity.  As mentioned already, part (i) of the theorem is well-known, see also the discussion after the theorem.  The part of the theorem stating that entanglement of states exists in non-commutative \Cs s is also well-known, see Takesaki, \cite[Theorem 4.14]{Takesaki:Vol-I} where it is shown that $S_*(\cA \otimes \cB) = S(\cA \otimes \cB)$ if and only if one of 
$\cA$ and $\cB$ is commutative.

\begin{theorem} \label{thm:twoabelian}
Let $\cA$ and $\cB$ be unital \Cs s. 
\begin{enumerate}
\item $S(\cA)$ is a Choquet simplex if and only if $\cA$ is commutative.
\item If one of $\cA$ and $\cB$ is commutative, then we have equalities in \eqref{eq:inclusions}, i.e.,
$$S(\cA) \otimes_* S(\cB) = S_*(\cA \otimes \cB) =S(\cA \otimes \cB) = S(\cA \otimes_\mx \cB) = S(\cA) \otimes^* S(\cB).$$
\item If both $\cA$ and $\cB$ are non-commutative, then the two inclusions 
$$S_*(\cA \otimes \cB) \subset S(\cA \otimes \cB), \qquad S(\cA \otimes_\mx \cB) \subset S(\cA) \otimes^* S(\cB),$$ 
in \eqref{eq:inclusions} are strict. In particular, $S(\cA) \otimes_* S(\cB) \ne S(\cA) \otimes^* S(\cB)$.
\item  One  of $\cA$ and $\cB$ is commutative if and only if $\cA^+ \otimes \, \cB^+= (\cA \otimes \cB)^+$ if and only if 
$\cA^+ \otimes_\mx \, \cB^+ = (\cA \otimes_\mx \cB)^+$.
\end{enumerate}
\end{theorem}

\begin{proof} If $\cA$ is commutative with spectrum $X$, then $S(\cA) = \mathrm{Prob}(X)$ is a Bauer simplex, and hence in particular a Choquet simplex. By Theorem~\ref{thm:N-P},  $S(\cA) \otimes_* S(\cB) =S(\cA) \otimes^* S(\cB)$ if one of $\cA$ and $\cB$ is commutative,  which in turns implies that (ii) holds, cf.\ \eqref{eq:inclusions}. 

As remarked below Definition~\ref{def:kappa}, if $\kappa(\cA,\cB)>1$, then $S(\cA \otimes_\mx \cB) \ne S(\cA) \otimes^* S(\cB)$. Together with Proposition~\ref{prop:kappa} this shows that $S(\cA \otimes_\mx \cB) \ne S(\cA) \otimes^* S(\cB)$, and hence that $S(\cA) \otimes_* S(\cB) \ne S(\cA) \otimes^* S(\cB)$, when both $\cA$ and $\cB$ are non-commutative. By Theorem~\ref{thm:N-P}, this also implies that $S(\cA)$ is not a Choquet simplex when $\cA$ is non-commutative, thus completing the proof of (i). 

Strictness of the inclusion  $S_*(\cA \otimes \cB) \subset S(\cA \otimes \cB)$, when both $\cA$ and $\cB$ are non-commutative, is contained in \cite[Theorem 4.14]{Takesaki:Vol-I}, as remarked above. While the proof in  \cite{Takesaki:Vol-I} likely is the most direct approach to this result, we nonetheless remark that we also can obtain this fact by combining results obtained here as follows: Non-commutativity of $\cA$ and $\cB$ implies that both \Cs s have sub-quotients isomorphic to $M_n(\C)$, for some $n \ge 2$, cf.\ Glimm's Lemma (Lemma~\ref{lm:Glimm}), and matrix algebras have entangled states, cf.\ Remark~\ref{rem:entangled-matrices}. Entanglement in sub-quotients lifts to entanglement in $\cA \otimes \cB$ by Proposition~\ref{prop:entanglement-quotient} and
Proposition~\ref{prop:inclusion}.

To prove (iv), suppose first that $\cA$ is commutative, in which case we may assume that $\cA = C(X)$, for some compact Hausdorff space $X$. Then $\cA \otimes \cB = C(X,\cB)$, and each $f \in C(X,\cB)$ can be approximated by elements of the form 
$f_0=\sum_{j=1}^n \varphi_j \otimes b_j$, for suitable $\varphi_j \in C(X)$ and $b_j = f(x_j)$, where 
$0 \le \varphi_j \le 1$, $\sum \varphi_j = 1$, and  $x_j \in X$. If $f$ is positve, then 
$f_0 \in C(X)^+ \otimes \cB^+$.

Suppose next that both $\cA$ and $\cB$ are non-commutative, both with rank $\ge n$, for some $n \ge 2$. Let $\cA_0 \subseteq \cA$ and $\cB_0 \subseteq \cB$ be as in the proof of Proposition~\ref{prop:kappa} (ii), and choose surjective \sh s $\pi_1 \colon \cA_0 \to M_n$ and $\pi_2 \colon \cB_0 \to M_n$. Then $\pi_1 \otimes \pi_2 \colon \cA_0 \otimes \cB_0 \to M_n \otimes M_n$ is a surjective \sh. (All involved \Cs s are nuclear, so the max and the min norms are the same.) 
We know that $H = H^{(n)} \in (M_n(\C) \otimes M_n(\C))^+$ is positive but does not belong to 
$M_n^+ \odot M_n^+ = M_n^+ \otimes M_n^+$, cf.\ Proposition~\ref{prop:max-tensor-matrix} and Remark~\ref{rem:S}. Lift $H$ along $\pi_1 \otimes \pi_2$ to a positive element $\tilde{H}$ in $\cA_0 \otimes \cB_0$.  Then $\tilde{H}$ does not belong to $\cA_0^+ \otimes \cB_0^+$ (since $\pi_1 \otimes \pi_2$ maps $\cA_0^+ \otimes \cB_0^+$ into $M_n^+ \otimes M_n^+$). It follows  Proposition~\ref{prop:regular-positive} (ii) that $\tilde{H}$ does not beong to $\cA^+ \otimes \cB^+$ either. The same conclusion holds with respect to the maximal tensor product using Proposition~\ref{prop:regular-positive} (iii).
\end{proof}

\begin{remark} \label{rem:aodotb}
It is an immediate consequence of the second part of Theorem~\ref{thm:twoabelian} (iii) above (or of part (iv)) that $\cA^+ \odot \cB^+ \subsetneq (\cA \odot \cB)^+$, when both $\cA$ and $\cB$ are non-commutative. This inclusion can be strict also in cases where both $\cA$ and $\cB$ are commutative, e.g., with $\cA = \cB = C([0,1])$. 
Indeed, with the  identification $C([0,1]) \otimes C([0,1]) = C([0,1]^2)$, the function $f(x,y) = x^2+y^2-2xy$ belongs to $(C([0,1]) \odot C([0,1]))^+$ but not to $C([0,1])^+ \odot C([0,1])^+$. One way of seeing the latter claim is to observe that the zero-set of any function $g$ in $C([0,1])^+ \odot C([0,1])^+$ must be a finite union of sets of the form $I \times J$, where $I,J \subseteq [0,1]$ are closed sets. 
\end{remark}

\noindent We included Theorem~\ref{thm:twoabelian} (i) in order  to put this classical fact into the context of entanglement and tensor products of compact convex sets.  For comparison, we present below a more traditional  route to this result.  Recall  for this  purpose that a compact convex set $K$ is a Choquet simplex if and only if the ordered vector space $A(K)$ of its affine functions has the Riesz Interpolation Property, 
\cite[Section II.3]{Alf:convex}. Accordingly, the state space $S(\cA)$ of a unital \Cs{} $\cA$ is a Choquet simplex if and only if $\cA_\sa$, equipped with the usual $C^*$-order, has the Riesz Interpolation Property. 

\begin{proposition} If $\cA$ is a non-commutative \Cs, then $\cA_\sa$ equipped with the usual order does not satisfy the Riesz Interpolation Property.
\end{proposition}

\noindent This result is contained in Stratila--Zsido, \cite[Section 4.8, Corollary 1]{Stratila-Zsido:Banach}, see also Kadison, \cite{Kadison:anti-lattice}, where it is shown that $\cA_\sa$ is an anti-lattice when $\cA$ is non-commutative. We give here a different proof, using Glimm's lemma.

\begin{proof} Let $\cA$ be a non-commutative \Cs. We find self-adjoint elements $a_1, a_2, b_1, b_2$ in $\cA$ such that $a_i \le b_j$, for $i,j=1,2$, and for which there is no self-adjoint $c \in \cA$ satisfying $a_i \le c \le b_j$, for for $i,j=1,2$. 

By Glimm's Lemma (Lemma~\ref{lm:Glimm}), there is a non-zero \sh{} $\varphi \colon CM_2 \to \cA$. Let $(e_{ij})$ denote the standard matrix units for $M_2(\C) \, (=M_2)$. Set  
$$f = t \otimes \big({\textstyle{-\frac23}} (e_{11}+e_{22}) + (e_{12}+e_{21})\big) \in CM_2,$$
and set $a_1=0$, $a_2 = \varphi(f)$, $b_j = \varphi(t \otimes e_{jj})$, for $j=1,2$, where, for each $a \in M_2(\C)$ we let $t \otimes a$ denote the function $t \mapsto ta$ in $CM_2$. 

We claim that  $a_2 \le b_j$, for $j=1,2$, while $a_2 \nleq 0$. To this end, check first that
$$ \begin{pmatrix} -2/3 & \; 1 \\ \; 1 & -2/3\end{pmatrix} \; \le \;  \begin{pmatrix} 1 & 0\\ 0 & 0\end{pmatrix}, \quad 
\begin{pmatrix} -2/3 & \; 1 \\ \; 1 & -2/3\end{pmatrix} \; \le \;  \begin{pmatrix} 0 & 0 \\ 0 & 1 \end{pmatrix}, \quad
\begin{pmatrix} -2/3 & \; 1 \\ \; 1 & -2/3\end{pmatrix} \; \nleq \; 0.
$$
Let $\pi_s \colon CM_2 \to M_2$ be evaluation at $s \in [0,1]$. Using the inequalities above,
we see that $\pi_s(f) \le \pi_s(t \otimes e_{jj})$ and $\pi_s(f) \nleq 0$, for $j=1,2$ and $s \in [0,1]$. This shows that $f \le t \otimes e_{jj}$ in $CM_2$, and hence that $a_2 \le b_j$, for $j=1,2$, and it also shows that $f$ is non-negative in each (non-zero) quotient of $CM_2$, and hence that $a_2$ is non-negative.

Suppose now that $a_i \le c \le b_j$, for $i,j=1,2$. Then  $0=a_1 \le c \le b_j$, for $j=1,2$, which implies $c=0$ (because $b_1$ and $b_2$ are orthogonal); but then we cannot have $a_2 \le c$. 
\end{proof}

\section{Tensor products of trace spaces of \Cs s} \label{sec:traces}

\noindent Let $\cA$ be a unital \Cs, and let $T(\cA)$ denote the set of tracial states on $\cA$. It is known that $T(\cA)$, if non-empty, is a  Choquet simplex, see \cite{Thoma:Gruppen} and \cite{BlaRor:trace-simplex}.  We refer to  \cite{Haa:quasitrace} and \cite{Pop:traces} for complete chacterizations of the \Cs s $\cA$ for which $T(\cA) \ne \emptyset$. The extreme points of $T(\cA)$  are ``factorial traces'', i.e., tracial states $\tau$ on $\cA$ for which $\pi_\tau(\cA)''$ is a factor. 

Let $i[\cA_\sa,\cA_\sa]$ denote the closed subspace of $\cA_\sa$ spanned by commutators $i(xy-yx)$, with $x,y \in \cA_\sa$, and let $q \colon \cA_\sa \to \cA_\sa/i[\cA_\sa,\cA_\sa]$ denote the quotient mapping. (See also \cite{CuntzPed-JFA}, where traces were studied using this space realized as $\cA^q$, a quotient of $\cA_\sa$ with respect to a specific equivalence relation.) A bounded hermitian functional on a \Cs{} $\cA$ is tracial if and only if it factors through $q$; for such a functional $\tau$, let $\widehat{\tau}$ denote its descent to $q(\cA_\sa)$.
Identify  $A(T(\cA))$ with $q(\cA_\sa)$ to obtain
$$T(\cA) \cong S(q(\cA_\sa), q(\cA^+), q(1_\cA)) =  \{\widehat{\tau} : \tau \in T(\cA)\}.$$

The tensor products of the trace simplexes of two \Cs s $\cA$ and $\cB$ are accordingly given by
\begin{eqnarray}
T(\cA) \otimes_* T(\cB) &=& S\big( q(\cA_\sa) \odot q(\cB_\sa), q(\cA^+) \otensorhat  q(\cB^+) , q(1_\cA) \otimes q(1_\cB)\big), \\
T(\cA) \otimes^* T(\cB) &=& S\big( q(\cA_\sa) \odot q(\cB_\sa), q(\cA^+) \odot q(\cB^+) ,q( 1_\cA) \otimes q(1_\cB)\big).
\end{eqnarray}
Since $T(\cA)$ and $T(\cB)$ are Choquet simplexes, the two tensor products above agree, cf.\ Theorem~\ref{thm:N-P}, and we denote them both by $T(\cA) \otimes T(\cB)$.

A linear functional on $\cA_\sa \odot \cB_\sa$ descends to a linear functional on $q(\cA_\sa) \odot q(\cB_\sa)$ if and only if it vanishes on the kernel of $q \otimes q \colon \cA_\sa \odot \cB_\sa \to q(\cA_\sa) \odot q(\cB_\sa)$ if and only if it is tracial. For the latter, use that $[a \otimes b, c \otimes d] = ca \otimes [b,d] + [a,c] \otimes bd$, when $a,c \in \cA$ and $b,d \in \cB$, to see that
$$\mathrm{ker}(q \otimes q) = i[\cA_\sa, \cA_\sa] \odot \cB_\sa + \cA_\sa \odot i[\cB_\sa, \cB_\sa] = i[\cA_\sa \odot \cB_\sa, \cA_\sa \odot \cB_\sa].$$
This shows that each pair $\tau_\cA \in T(\cA)$ and $\tau_\cB \in T(\cB)$ yields a  linear functional $\widehat{\tau}_\cA \otimes \widehat{\tau}_\cB$ on $q(\cA_\sa) \odot q(\cB_\sa)$. Conversely, each trace in $T(\cA \otimes \cB)$ or in $T(\cA \otimes_\mx \cB)$ restricts to a linear functional on $\cA_\sa \odot \cB_\sa$ which further descends to a linear functional on $q(\cA_\sa) \odot q(\cB_\sa)$.

In analogy with the corresponding definition for state spaces, consider the set 
$$T_*(\cA \otimes \cB) = \overline{\conv\{\tau_\cA \otimes \tau_\cB : \tau_\cA \in T(\cA), \tau_\cB \in T(\cB)\}},$$
which is a weak$^*$-closed convex subset of $T(\cA \otimes \cB)$, the trace simplex of the minimal tensor product $\cA \otimes \cB$.

\begin{theorem} \label{thm:traces}
Let $\cA$ and $\cB$ be unital \Cs s both admitting at least one tracial state. Then
$$T(\cA) \otimes T(\cB) = T_*(\cA \otimes \cB) = T(\cA \otimes \cB) = T(\cA \otimes_\mx \cB).$$
\end{theorem}

\begin{proof} We  show that 
$$T(\cA) \otimes_* T(\cB) = T_*(\cA \otimes \cB) \subseteq T(\cA \otimes \cB) \subseteq T(\cA \otimes_\mx \cB) \subseteq T(\cA) \otimes^* T(\cB),$$
and then use that $T(\cA) \otimes_* T(\cB) = T(\cA) \otimes^* T(\cB)$,  cf.\ Theorem~\ref{thm:N-P},   to finish the proof. 

By Proposition~\ref{prop:tensor-I} (iii), and following the discussion above,
$$T(\cA) \otimes_* T(\cB) =\overline{\conv} \{\widehat{\tau}_\cA \otimes \widehat{\tau}_\cB : \tau_\cA \in 
T(\cA), \; \tau_\cB \in T(\cB)\} = T_*(\cA \otimes \cB),$$
where the last equality is via the descent map $\tau_\cA \otimes \tau_\cB \mapsto \widehat{\tau}_\cA \otimes \widehat{\tau}_\cB$. 

The inclusion $T_*(\cA \otimes \cB) \subseteq T(\cA \otimes \cB)$ holds by definition, and  $T(\cA \otimes \cB) \subseteq T(\cA \otimes_\mx \cB)$  because $\cA \otimes \cB$ is a quotient of $\cA \otimes_\mx \cB$.

Finally, as in the proof of Lemma~\ref{lm:max-states}, 
we can identify traces on $\cA \otimes_\mx \cB$ with normalized tracial functionals on $\cA_\sa \odot \cB_\sa$ which are positive on $(\cA \odot \cB)^+$. Since $(\cA \odot \cB)^+ \supseteq \cA^+ \odot \cB^+$, we obtain the last inclusion above.
\end{proof}

\noindent The identities $T_*(\cA \otimes \cB)=T(\cA \otimes \cB) = T(\cA \otimes_\mx \cB)$ appearing in Theorem~\ref{thm:traces} are well-known.  In our context, the former equality says that traces cannot be entangled! The latter equality was observed by Kirchberg in \cite{Kir:CEP} (included in his proof of (B4) $\Rightarrow$ (B3)), see also \cite[Exercise 13.3.3]{BO-Approx} and \cite{KirRor:Central-sequence}. It can be rephrased by saying that every tracial state on $\cA \otimes_\mx \cB$ factors through $\cA \otimes \cB$, i.e., vanishes on the kernel of the canonical map $\cA \otimes_\mx \cB \to \cA \otimes \cB$. 

One can prove the two identities above directly as follows:  It suffices to show that $\partial_e T(\cA \otimes_\mx \cB) \subseteq T_*(\cA \otimes \cB)$. 
Let $\tau \in \partial_e T(\cA \otimes_\mx \cB)$, and let $\tau_\cA \in T(\cA)$ and $\tau_\cB \in T(\cB)$ be the ``marginal distributions'' of $\tau$, cf.\ \eqref{eq:slice} and \eqref{eq:pi_j}. As $\tau$ is factorial (being extremal), it follows that $\pi_\tau(\cA \otimes 1_\cB)'' \cong \pi_{\tau_\cA}(\cA)''$ and $\pi_\tau(1_\cA \otimes \cB)'' \cong \pi_{\tau_\cB}(\cB)''$ are factors. Hence $\tau_\cA$ and $\tau_\cB$ are factorial and therefore extremal. But then $\tau = \tau_\cA \otimes \tau_\cB \in T_*(\cA \otimes \cB)$ by Lemma~\ref{lm:slice}.

Theorem~\ref{thm:traces} provides the new information that $T(\cA) \otimes T(\cB) =  T(\cA \otimes \cB)$. It also offers a new way of understanding the identities $T_*(\cA \otimes \cB)=T(\cA \otimes \cB) = T(\cA \otimes_\mx \cB)$ in terms of entanglement and tensor products of compact convex sets.

Combining Theorem~\ref{thm:traces} with Example~\ref{ex:tensor} we get the corollary below, valid also for the max-tensor product of the \Cs s. 

\begin{corollary} \label{cor:traces} Let $\cA$ and $\cB$ be unital \Cs s with non-empty trace simplexes. Then $T(\cA \otimes \cB)$ is a Bauer simplex if and only if both $T(\cA)$ and $T(\cB)$ are Bauer simplexes in which case $\partial_e T( \cA \otimes \cB) = \partial_e T(\cA) \times \partial_e T(\cB)$.
\end{corollary}

\noindent As before, we identify $\partial_e T(\cA) \times \partial_e T(\cB)$ with the set of extremal product traces $\tau_\cA \otimes \tau_\cB$.
Using the notion of infinite tensor products of simplexes defined at the end of Section~\ref{sec:NP} we can extend the results above to infinite tensor products (also valid for the max-tensor product of the \Cs s).

\begin{theorem} \label{thm:traces-II} Let $\cA_1,\cA_2, \cA_3, \dots$ be a sequence of unital \Cs s each with non-empty trace simplex. Then $T(\bigotimes_{k \ge 1} \cA_k) = \bigotimes_{k \ge 1} T(\cA_k)$.
\end{theorem}

\begin{proof} The infinite tensor product $\bigotimes_{k \ge 1} \cA_k$ is the inductive limit $\varinjlim \bigotimes_{k = 1}^n \cA_k$, where each $\bigotimes_{k = 1}^n \cA_k$ is embedded in $\bigotimes_{k = 1}^{n+1} \cA_k$
via the map $x \mapsto x \otimes 1_{\cA_{n+1}}$. It follows that the trace simplex of $\bigotimes_{k \ge 1} \cA_k$  is the inverse limit
$$\xymatrix{T(\cA_1) & T(\cA_1 \otimes \cA_2) \ar[l] & T(\cA_1 \otimes \cA_2 \otimes \cA_3) \ar[l] & \cdots \ar[l] & T\big(\bigotimes_{k \ge 1} \cA_k\big). \ar[l]
}$$
By Theorem~\ref{thm:traces}, for each $n \ge 1$,  $T(\bigotimes_{k=1}^n \cA_k) = \bigotimes_{k=1}^n T(\cA_k)$,  and the connecting map $\bigotimes_{k=1}^{n+1} T(\cA_k) \to \bigotimes_{k=1}^{n} T(\cA_k)$ induced by the connecting map $\bigotimes_{k = 1}^n \cA_k \to \bigotimes_{k = 1}^{n+1} \cA_k$ is the precisely the map $\pi_n$ from \eqref{eq:inf-tensor}. This proves the claim.
\end{proof}

\noindent The result below follows from Theorem~\ref{thm:traces-II} and Example~\ref{ex:inftensor}, and is valid also for the max-tensor product on the \Cs s.

\begin{corollary} \label{cor:traces-II} Let $\cA_1,\cA_2, \cA_3, \dots$ be a sequence of unital \Cs s each with non-empty trace simplex. Then $T(\bigotimes_{k \ge 1} \cA_k)$ is a Bauer simplex if and only if $T(\cA_k)$ is a Bauer simplex, for each $k$, in which case $\partial_e T(\bigotimes_{k \ge 1} \cA_k) = 
\prod_{k=1}^\infty \partial_e T(\cA_k)$.
\end{corollary}

\noindent  By primeness of the Poulsen simplex, Corollary~\ref{cor:not-Poulsen}, we infer that if $\cA$ and $\cB$ are unital \Cs s both admitting tracial states, then $T(\cA \otimes \cB)$ is the Poulsen simplex if and only if one of $T(\cA)$ and $T(\cB)$ is trivial and the other is the Poulsen simplex. In particular if $\Gamma_1$ and $\Gamma_2$ are discrete groups with more than one element, then $T(C^*(\Gamma_1 \times \Gamma_2))$ is never a Poulsen simplex, since $T(C^*(\Gamma))$ is non-empty and non-trivial for all groups $\Gamma$ with more than one element. This  relates to recent works, \cite{OSV:Poulsen, ISV:Poulsen}, showing that the trace simplex of the full group \Cs{} $C^*(\Gamma)$ is the Poulsen simplex for a large class of groups $\Gamma$, including free groups, and that $T(\cA)$ likewise is the Poulsen simplex for many universal free product \Cs s.

{\small{
\bibliographystyle{amsplain}
\providecommand{\bysame}{\leavevmode\hbox to3em{\hrulefill}\thinspace}
\providecommand{\MR}{\relax\ifhmode\unskip\space\fi MR }
\providecommand{\MRhref}[2]{%
  \href{http://www.ams.org/mathscinet-getitem?mr=#1}{#2}
}
\providecommand{\href}[2]{#2}

}}

\vspace{1cm}

\noindent
\begin{tabular}{ll}
Magdalena Musat & Mikael R\o rdam \\
Department of Mathematical Sciences & Department of Mathematical Sciences\\
University of Copenhagen & University of Copenhagen\\ 
Universitetsparken 5, DK-2100, Copenhagen \O & Universitetsparken 5, DK-2100, Copenhagen \O \\
Denmark & Denmark\\
musat@math.ku.dk & rordam@math.ku.dk
\end{tabular}

\end{document}